\documentclass[a4paper,12pt]{amsart}
\usepackage{enumerate,stmaryrd}
\usepackage{amsmath,amssymb,amsthm,xspace,amscd,latexsym,amscd}
\usepackage{xy}
\xyoption{all}
\usepackage[T1]{fontenc}

\newtheorem{theorem}{Theorem}%[section]
\newtheorem{lemma}[theorem]{Lemma}%[section]
\newtheorem{proposition}[theorem]{Proposition}%[section]
\newtheorem{corollary}[theorem]{Corollary}%[section]

\newtheorem{remark}[theorem]{Remark}%[section]

\begin{document}

\title{Classification of simple $\mathfrak{q}_{2}$-supermodules}
\author{Volodymyr Mazorchuk}
\date{\today}
\begin{abstract}
We classify all simple supermodules over the queer Lie superalgebra
$\mathfrak{q}_{2}$ up to classification of equivalence classes of
irreducible elements in a certain Euclidean ring.   
\end{abstract}

\maketitle

\section{Introduction and description of the results}\label{s1} 

The Lie algebra $\mathfrak{sl}_2(\mathbb{C})$ is the only complex
semisimple finite dimensional Lie algebra for which all simple 
(not necessarily  finite dimensional) modules are classified
(by Block, see \cite{Bl}). This classification was later extended 
to all generalized Weyl algebras by Bavula, see \cite{Ba}. The approach
via generalized Weyl algebra further allowed to obtain a classification  
of all supermodules over the Lie superalgebra $\mathfrak{osp}(1,2)$, 
see \cite{BO}. All the above classifications are given up to classification 
of  equivalence classes of irreducible elements in a certain Euclidean ring
(see \cite{Ba} or \cite[Chapter~6]{Ma} for details).

The queer Lie superalgebra $\mathfrak{q}_{2}$ is another Lie superalgebra,
which is closely related to the Lie algebra $\mathfrak{sl}_2(\mathbb{C})$.
It would be more correct to say that $\mathfrak{q}_{2}$ is closely related 
to the Lie algebra $\mathfrak{gl}_2(\mathbb{C})$. Namely, the superalgebra
$\mathfrak{q}_{2}$ can be regarded a a kind of a ``super-doubling'' of
the algebra $\mathfrak{gl}_2(\mathbb{C})$ in the sense that 
$\mathfrak{q}_{2}$ is a direct sum of one even and one
odd copy of $\mathfrak{gl}_2(\mathbb{C})$. Although all queer Lie
superalgebras  $\mathfrak{q}_{n}$ are classical, their properties 
are rather  different from the properties of other classical 
Lie superalgebras. For example, the Cartan subsuperalgebra of 
$\mathfrak{q}_{n}$ (in particular, of $\mathfrak{q}_{2}$) is not commutative,
which makes the corresponding theory of weight supermodules more complicated, 
but also more interesting. In comparison to Lie algebras, there are 
several types of degenerations in the representation theory of 
Lie superalgebras, for example there are atypical, typical and 
strongly typical types of supermodules, which are also subdivided into
regular and singular subtypes. All these degenerations make representation
theory of Lie superalgebras much more complicated, but again, more interesting. 
Various classes of representations of 
$\mathfrak{q}_{n}$ were studied in \cite{Pe2,PS,BK,Br,Go2,Fr,FM,Or}.

As $\mathfrak{gl}_2(\mathbb{C})$ is a direct sum of 
$\mathfrak{sl}_2(\mathbb{C})$ and a one-dimensional central subalgebra,
the classification of simple $\mathfrak{sl}_2(\mathbb{C})$-modules
extends to a classification of simple $\mathfrak{gl}_2(\mathbb{C})$-modules
in the natural way. This naturally raises the question whether the 
classification of simple $\mathfrak{gl}_2(\mathbb{C})$-modules can be
extended to a classification of simple $\mathfrak{q}_{2}$-supermodules.
A strong evidence for a close connection between simple (super)modules over 
(any classical) Lie superalgebra and its even Lie subalgebra was obtained
by Penkov in \cite{Pe}. In the present paper we address this question
and prove the following main result, which, in particular, 
reduces the classification of  all simple $\mathfrak{q}_{2}$-supermodules 
to that of all simple $\mathfrak{gl}_2$-modules
(see Section~\ref{s2} for the setup):

\begin{theorem}\label{mthm1}
\begin{enumerate}[(i)]
\item\label{mthm1.1} Every simple  $\mathfrak{q}_{2}$-supermodule
has finite length as a $\mathfrak{gl}_2$-module.
\item\label{mthm1.2} For every primitive ideal $\mathcal{I}$ of
$U(\mathfrak{q}_{2})$ there is an explicitly given  primitive 
ideal $\mathtt{I}$ of $U(\mathfrak{gl}_2)$ and an explicitly given 
$U(\mathfrak{q}_{2})\text{-}U(\mathfrak{gl}_2)$-bimodule $B$ such 
that the functor $B\otimes_{U(\mathfrak{gl}_2)}{}_-$ induces a bijection
from the set of isomorphism classes of simple 
$\mathfrak{gl}_2$-modules annihilated by $\mathtt{I}$ to the set of
isomorphism classes (up to parity change) of simple 
$\mathfrak{q}_2$-supermodules annihilated by $\mathcal{I}$. 
\item\label{mthm1.3} For every strongly typical or atypical simple 
$\mathfrak{q}_{2}$-supermodule $N$ we have $N\not\cong \Pi N$, where 
$\Pi$ denotes the parity change functor. For every typical simple 
$\mathfrak{q}_{2}$-supermodule $N$, which is not strongly typical,
we have $N\cong \Pi N$.
\end{enumerate}
\end{theorem}

The $U(\mathfrak{q}_{2})\text{-}U(\mathfrak{gl}_2)$-bimodule $B$, 
which appears in the formulation of  Theorem~\ref{mthm1} is a 
Harish-Chandra $U(\mathfrak{q}_{2})\text{-}U(\mathfrak{gl}_2)$-bimodule.
Such bimodules are our main technical tool and a substantial part of
the paper is devoted to developing the corresponding techniques.
This also makes most of our arguments quite homological.
In addition to Theorem~\ref{mthm1} we also explicitly describe the
{\em rough structure} of all simple $\mathfrak{q}_{2}$-supermodules,
considered as $\mathfrak{gl}_2$-modules, in the sense of
\cite{KM,MS}. In the atypical and regular typical cases the rough 
structure coincides with the actual $\mathfrak{gl}_2$-module structure, 
whereas for singular typical supermodules the difference
between these two structures is a potential direct sum of finitely many
copies of one-dimensional $\mathfrak{gl}_2$-modules, which we are not
able to determine explicitly.

The structure of the paper is as follows: We collect all necessary
preliminaries, in particular on primitive ideals in 
$U(\mathfrak{q}_{2})$ and on Harish-Chandra bimodules,
in Section~\ref{s2}. In Section~\ref{s3} we prove
Theorem~\ref{mthm1} and describe the rough structure of 
simple $\mathfrak{q}_{2}$-supermodules. We also outline an
alternative approach to classification of simple weight supermodules
via a localized superalgebra. In Subsections~\ref{s4.10}--\ref{s4.12}
we extend Theorem~\ref{mthm1} to superalgebras $\mathfrak{pq}_2$,
$\mathfrak{sq}_2$ and $\mathfrak{psq}_2$. 
\vspace{2mm}

\noindent
{\bf Acknowledgments:} The work was partially supported by the
Swe\-dish Research Council. The author thanks Ken Brown for
useful discussions.
\vspace{2mm}

\section{The superalgebra $\mathfrak{q}_{2}$
and $\mathfrak{q}_{2}$-supermodules}\label{s2} 

\subsection{The superalgebra $\mathfrak{q}_{2}$}\label{s2.1} 

For all undefined notions we refere the reader to \cite{Fr}.
Let $\mathbb{Z}$, $\mathbb{N}$ and $\mathbb{N}_0$ denote the sets of 
all, positive  and nonnegative integers, respectively. Let $\Bbbk$ 
be an uncountable algebraically closed field of characteristic zero.
Denote by $\mathbf{i}\in \Bbbk$ a fixed square root of $-1$. Let
$\mathtt{g}=\mathfrak{gl}_2(\Bbbk)$ denote  the general linear Lie algebra 
of $2\times 2$  matrices over $\Bbbk$. The {\em queer Lie superalgebra} 
$\mathfrak{q}=\mathfrak{q}_2$ over $\Bbbk$ consists of all block matrices 
of the form 
\begin{displaymath}
\mathtt{M}(A,B)=\left(\begin{array}{cc}A&B\\B&A\end{array}\right),\quad
A,B\in \mathfrak{gl}_2.
\end{displaymath}
The even and the odd spaces $\mathfrak{q}_{\overline{0}}$ and
$\mathfrak{q}_{\overline{1}}$ consist of the  matrices $\mathtt{M}(A,0)$
and $\mathtt{M}(0,B)$, respectively, and we have 
$\mathfrak{q}=\mathfrak{q}_{\overline{0}}\oplus\mathfrak{q}_{\overline{1}}$.
For a homogeneous element $X\in \mathfrak{q}$ we denote by 
$\overline{X}\in\mathbb{Z}/2\mathbb{Z}$ the degree of $X$. 
Then the Lie superbracket in $\mathfrak{q}$ is given by  
$[X,Y]=XY-(-1)^{\overline{X}\overline{Y}}YX$, where $X,Y\in \mathfrak{q}$ 
are homogeneous.

For $i,j\in\{1,2\}$ let $E_{ij}\in \mathfrak{gl}_n$ denote the
corresponding matrix unit. Set 
\begin{gather*}
E=\mathtt{M}(E_{12},0), F=\mathtt{M}(E_{21},0), 
H_1=\mathtt{M}(E_{11},0), H_2=\mathtt{M}(E_{22},0);\\
\overline{E}=\mathtt{M}(0,E_{12}), \overline{F}=\mathtt{M}(0,E_{21}), \overline{H}_1=\mathtt{M}(0,E_{11}), \overline{H}_2=\mathtt{M}(0,E_{22}).
\end{gather*}
We have the {\em Cartan subsuperalgebra} $\mathfrak{h}$ of 
$\mathfrak{q}$, which is the linear span of  $H_1$, $H_2$,
$\overline{H}_1$ and $\overline{H}_2$.  The superalgebra 
$\mathfrak{h}$ inherits from  $\mathfrak{q}$ the  decomposition 
$\mathfrak{h}= \mathfrak{h}_{\overline{0}}\oplus\mathfrak{h}_{\overline{1}}$.
Similarly we define the subsuperalgebra
$\mathfrak{n}= \mathfrak{n}_{\overline{0}}\oplus\mathfrak{n}_{\overline{1}}$,
which is generated by $E$ (it spans $\mathfrak{n}_{\overline{0}}$) 
and $\overline{E}$ (it spans $\mathfrak{n}_{\overline{1}}$); and the 
subsuperalgebra
$\mathfrak{m}= \mathfrak{m}_{\overline{0}}\oplus\mathfrak{m}_{\overline{1}}$,
which is generated by $F$ (it spans $\mathfrak{m}_{\overline{0}}$) 
and $\overline{F}$ (it spans $\mathfrak{m}_{\overline{1}}$). This leads to 
the standard triangular decomposition 
$\mathfrak{q}=\mathfrak{m}\oplus\mathfrak{h}\oplus\mathfrak{n}$.
The even Lie subalgebra $\mathfrak{q}_{\overline{0}}$, which is the linear 
span of $E$, $F$, $H_1$ and $H_2$, is identified with the Lie 
algebra $\mathtt{g}$ in the obvious way.

Let $U(\mathfrak{q})$ and $U(\mathtt{g})$ denote the universal enveloping
(super)algebras of $\mathfrak{q}$ and $\mathtt{g}$, respectively.
Let $Z(\mathfrak{q})$ and $Z(\mathtt{g})$ denote the centers of 
the algebras $U(\mathfrak{q})$ and $U(\mathtt{g})$, respectively.
The PBW theorem for Lie superalgebras (see \cite{Ro}) asserts that
$U(\mathfrak{q})$ is free of finite rank over $U(\mathtt{g})$ both 
as a right and as a left module with the basis 
\begin{displaymath}
\big\{a^{\varepsilon_1}b^{\varepsilon_2}b^{\varepsilon_3}
d^{\varepsilon_4}:
\varepsilon_1,\varepsilon_2,\varepsilon_3,\varepsilon_4\in\{0,1\}\big\},
\end{displaymath}
where $\{a,b,c,d\}=\{\overline{E},\overline{F},\overline{H}_1,
\overline{H}_2\}$.

\subsection{Supermodules}\label{s2.2} 

If $\mathfrak{a}=\mathfrak{a}_{\overline{0}}\oplus
\mathfrak{a}_{\overline{1}}$ is a Lie superalgebra over $\Bbbk$, 
then an $\mathfrak{a}$ {\em supermodule} is a $\Bbbk$-vector  superspace
$V=V_{\overline{0}}\oplus V_{\overline{1}}$ and a Lie superalgebra
homomorphism from $\mathfrak{a}$ to the Lie superalgebra of 
all linear operators on $V$. 

A homomorphism $\varphi:V\to W$ of two $\mathfrak{a}$-supermodules
$V$ and $W$ is a homogeneous linear map {\em of degree zero} from
$V$ to $W$, which intertwines the actions of $\mathfrak{a}$ on 
$V$ and $W$. We denote by $\mathfrak{a}\text{-}\mathrm{sMod}$ the category
of all $\mathfrak{a}$-supermodules with the above defined morphisms.
We  will use the standard notation and denote morphisms in 
$\mathfrak{a}\text{-}\mathrm{sMod}$ by $\mathrm{Hom}_{\mathfrak{a}}$.

The category $\mathfrak{a}\text{-}\mathrm{sMod}$ is abelian with 
usual kernels and cokernels. Simple objects in
$\mathfrak{a}\text{-}\mathrm{sMod}$ are {\em simple} 
$\mathfrak{a}$-supermodules, that is $\mathfrak{a}$-\-su\-per\-mo\-du\-les,
which do not have proper subsupermodules. An example of a simple
$\mathfrak{a}$-supermodule is the {\em trivial} supermodule
$\Bbbk=\Bbbk_{\overline{0}}$, which is defined using the zero action
of $\mathfrak{a}$.

Let $\Pi$ denote the endofunctor of $\mathfrak{a}\text{-}\mathrm{sMod}$,
which changes the parity. For example, if $\Bbbk$ is the trivial 
$\mathfrak{a}$-supermodule from the previous paragraph (which is purely 
even), then $\Pi\Bbbk$ is purely odd. In particular, $\Bbbk$ and 
$\Pi\Bbbk$ are not isomorphic in $\mathfrak{a}\text{-}\mathrm{sMod}$.

\subsection{$\mathfrak{h}$-supermodules}\label{s2.3} 

Elements in $\mathfrak{h}^*_{\overline{0}}$ are called {\em weights}
and are written $\lambda=(\lambda_1,\lambda_2)$ with respect to 
the basis $(\varepsilon_1,\varepsilon_2)$ of $\mathfrak{h}^*_{\overline{0}}$,
which is dual to the basis $(H_1,H_2)$ of $\mathfrak{h}_{\overline{0}}$. 
We set $\alpha=(1,-1)\in\mathfrak{h}^*_{\overline{0}}$ 
(the positive root of $\mathfrak{q}$).

For every $z\in\Bbbk$ we fix some $\sqrt{z}$.
For $\lambda=(\lambda_1,\lambda_2)\in \mathfrak{h}^*_{\overline{0}}$
we define the $\mathfrak{h}$-supermodule $\mathcal{V}(\lambda)$ as
follows: The supermodule $\mathcal{V}(0)$ is the trivial supermodule $\Bbbk$.
If $\lambda\neq 0$, then the supermodule $\mathcal{V}(\lambda)$ has 
a one-dimensional even space, spanned by $v$, and a 
one-dimensional odd space, spanned by $\overline{v}$, such that the action
of the elements $H_1$, $H_2$, $\overline{H}_1$, $\overline{H}_1$ in the
basis $\{v,\overline{v}\}$ is given by the following formulae:
\begin{gather*}
H_1(v)=\lambda_1 v;\, H_2(v)=\lambda_2 v;\,
\overline{H}_1(v)=\sqrt{\lambda_1}\overline{v};\,
\overline{H}_2(v)=-\mathbf{i}\sqrt{\lambda_2}\overline{v};\\
H_1(\overline{v})=\lambda_1 \overline{v};\, 
H_2(\overline{v})=\lambda_2 \overline{v};\,
\overline{H}_1(\overline{v})=\sqrt{\lambda_1}v;\,
\overline{H}_2(\overline{v})=\mathbf{i}\sqrt{\lambda_2}v.
\end{gather*}
This is usually depicted as follows:
\begin{displaymath}
\xymatrix{ 
v \ar@/^1pc/[rrrrrr]^{\overline{H}_1=\sqrt{\lambda_1},\,\,\,
\overline{H}_2=-\mathbf{i}\sqrt{\lambda_2}}
\ar@(ld,lu)[]^{\begin{array}{c}H_1=\lambda_1\\H_2=\lambda_2\end{array}}
&&&&&& 
\overline{v}
\ar@/^1pc/[llllll]^{\overline{H}_1=\sqrt{\lambda_1},\,\,\,
\overline{H}_2=\mathbf{i}\sqrt{\lambda_2}}
\ar@(rd,ru)[]_{\begin{array}{c}H_1=\lambda_1\\H_2=\lambda_2\end{array}}
}
\end{displaymath}

\begin{lemma}\label{lem1} 
\begin{enumerate}[(i)]
\item \label{lem1.1} Every simple  $\mathfrak{h}$-supermodule
is isomorphic to either  $\mathcal{V}(\lambda)$ or
$\Pi \mathcal{V}(\lambda)$ for some 
$\lambda\in \mathfrak{h}^*_{\overline{0}}$.
\item \label{lem1.2} For $\lambda\in \mathfrak{h}^*_{\overline{0}}$
we have $\mathcal{V}(\lambda)\cong\Pi\mathcal{V}(\lambda)$
if and only if $\lambda_1\lambda_2=0$ and $\lambda\neq 0$.
\end{enumerate}
\end{lemma}

\begin{proof}
This is well-known and follows directly from the theory of Clifford 
algebras. See for example \cite[Appendix~A]{Go2} or, alternatively,
\cite{Or} for full details and a direct computational approach.
\end{proof}

\subsection{$\mathfrak{q}_{2}$-supermodules}\label{s2.4} 

A $\mathfrak{q}$-supermodule $V$ is called {\em weight} provided
that the action of $\mathfrak{h}_{\overline{0}}$ on $V$ is 
diagonalizable. This means that 
\begin{displaymath}
V=\bigoplus_{\lambda\in\mathfrak{h}^*_{\overline{0}}}V_{\lambda},\quad
\text{ where }\quad
V_{\lambda}=\{v\in V: H_1(v)=\lambda_1 v,\,H_2(v)=\lambda_2 v \}.
\end{displaymath}
Each $V_{\lambda}$ is obviously an $\mathfrak{h}$-subsupermodule.
If $V_{\lambda}$ is finite dimensional, then from Lemma~\ref{lem1}
it follows that $V_{\lambda}$ has a finite composition series with 
subquotients isomorphic to either $\mathcal{V}(\lambda)$ or 
$\Pi \mathcal{V}(\lambda)$. The category of all weight 
$\mathfrak{q}$-supermodules is obvious closed with respect to 
taking any subquotients and direct sums.

For a simple $\mathfrak{h}$-supermodule $V$ set $\mathfrak{n}V=0$
and define the {\em Verma} $\mathfrak{q}$-supermodule
$M(V)$ as follows:
\begin{displaymath}
M(V)=U(\mathfrak{q})\bigotimes_{U(\mathfrak{h}\oplus\mathfrak{n})}V.
\end{displaymath}
The supermodule $M(V)$ is a weight supermodule, it has 
a unique simple top, which we will denote by $L(V)$. From the
definitions it follows that $\Pi M(V)\cong M(\Pi V)$ and
$\Pi L(V)\cong L(\Pi V)$.

A weight $\lambda$ is called 
\begin{itemize}
\item {\em integral} provided that $\lambda_1-\lambda_2\in\mathbb{Z}$;
\item {\em strongly integral} provided that $\lambda_1,\lambda_2\in\mathbb{Z}$;
\item {\em dominant} provided that $\lambda_1-\lambda_2\in\mathbb{N}$;
\item {\em regular} provided that $\lambda_1\neq \lambda_2$;
\item {\em typical} provided that $\lambda_1+\lambda_2\neq 0$;
\item {\em strongly typical} provided that 
it is typical and $\lambda_1,\lambda_2\neq 0$.
\end{itemize}

\begin{lemma}\label{lem2} 
The supermodules $L(\mathcal{V}(\lambda))$ and $\Pi L(\mathcal{V}(\lambda))$,
where  $\lambda\in \mathfrak{h}^*_{\overline{0}}$ is either zero or 
dominant, constitute an exhaustive list of all simple 
finite dimensional  $\mathfrak{h}$-supermodules.
\end{lemma}

\begin{proof}
See, for example, \cite{Pe2} or, alternatively,
\cite{Or} for full details.
\end{proof}

The inclusion $\mathtt{g}\subset\mathfrak{q}$ of superalgebras
(here $\mathtt{g}$ is considered as a purely even superalgebra)
gives rise to the usual restriction functor 
$\mathrm{Res}_{\mathtt{g}}^{\mathfrak{q}}:
\mathfrak{q}\text{-}\mathrm{sMod}\to
\mathtt{g}\text{-}\mathrm{sMod}$. We identify 
$\mathtt{g}\text{-}\mathrm{Mod}$ with the full subcategory of 
$\mathtt{g}\text{-}\mathrm{sMod}$, consisting of even supermodules.
Hence we can compose $\mathrm{Res}_{\mathtt{g}}^{\mathfrak{q}}$ 
with the projection on 
$\mathtt{g}\text{-}\mathrm{Mod}$ and get the restriction functor 
$\mathrm{Res}:\mathfrak{a}\text{-}\mathrm{sMod}\to
\mathtt{g}\text{-}\mathrm{Mod}$.
We denote by $M^{\mathtt{g}}(\lambda)$ and $L^{\mathtt{g}}(\lambda)$
the Verma $\mathtt{g}$-module corresponding to $\lambda$ and its
unique simple quotient, respectively (see for example 
\cite[Chapter~3]{Ma}). Note that from the definitions it follows directly 
that  $\mathrm{Res}_{\mathtt{g}}^{\mathfrak{q}}\, 
M(V)\cong \mathrm{Res}_{\mathtt{g}}^{\mathfrak{q}}\, M(\Pi V)$ 
and $\mathrm{Res}_{\mathtt{g}}^{\mathfrak{q}}\, L(V)\cong 
\mathrm{Res}_{\mathtt{g}}^{\mathfrak{q}}\, L(\Pi V)$.

\begin{lemma}\label{lempar}
For any simple  $\mathfrak{h}$-supermodule $V\not\cong \Bbbk, \Pi\Bbbk$
we have isomorphisms $\mathrm{Res}\, M(V)\cong \mathrm{Res}\, M(\Pi V)$ 
and $\mathrm{Res}\, L(V)\cong \mathrm{Res}\, L(\Pi V)$.
\end{lemma}

\begin{proof}
Assume that $V$ is a weight $\lambda\neq 0$. We consider the element
$\overline{H}_1+\overline{H}_1\in U(\mathfrak{q})$. This element
commutes with every element in $U(\mathtt{g})$ and squares to
$H_1+H_2$. Hence for typical $\lambda$ multiplication with
$\overline{H}_1+\overline{H}_1$ defines a
$\mathtt{g}$-isomorphism between $M(V)_{\overline{0}}$
and $M(V)_{\overline{1}}$ (resp. $L(V)_{\overline{0}}$
and $L(V)_{\overline{1}}$) in both direction, and the claim follows.

For atypical $\lambda\neq 0$ from the definition of 
$V(\lambda)$ it follows that multiplication with
$\overline{H}_1+\overline{H}_1$ defines a
$\mathtt{g}$-isomorphism from either $V(\lambda)_{\overline{0}}$
to $V(\lambda)_{\overline{1}}$ or vice versa. Using the definition
of the Verma supermodule this isomorphism lifts to an isomorphism from
$M(V)_{\overline{0}}$ to $M(V)_{\overline{1}}$ (or vice versa, respectively)
and then induces an isomorphism from $L(V)_{\overline{0}}$
to  $L(V)_{\overline{1}}$ (or vice versa, respectively). Hence the 
claim follows in this case as well.
\end{proof}

Later on we will need the following explicit description of the
restrictions of the supermodules $M(V)$ and $L(V)$.

\begin{lemma}\label{lem3} 
\begin{enumerate}[(i)]
\item\label{lem3.1}  
Let $\lambda\in \mathfrak{h}^*_{\overline{0}}$, $\lambda\neq 0$, and $V\in
\{\mathcal{V}(\lambda),\Pi \mathcal{V}(\lambda)\}$. Then we have
an isomorphism $\mathrm{Res}\, M(V)\cong X$, where for the 
$\mathtt{g}$-module $X$ there is a short exact sequence
\begin{displaymath}
0\to M^{\mathtt{g}}(\lambda)\to X\to M^{\mathtt{g}}(\lambda-\alpha)\to 0.
\end{displaymath}
\item\label{lem3.2}
We have $\mathrm{Res}\, M(\Bbbk)\cong M^{\mathtt{g}}(0)$
and $\mathrm{Res}\, \Pi M(\Bbbk)\cong M^{\mathtt{g}}(-\alpha)$.
\end{enumerate}
\end{lemma}

\begin{proof}
This follows from the definitions and the PBW theorem.
\end{proof}

\begin{lemma}\label{lem4} 
Let $\lambda\in \mathfrak{h}^*_{\overline{0}}$ and $V\in
\{\mathcal{V}(\lambda),\Pi \mathcal{V}(\lambda)\}$. 
\begin{enumerate}[(i)]
\item\label{lem4.1} 
$L(\mathcal{V}(0))\cong \Bbbk$, $L(\Pi\mathcal{V}(0))\cong \Pi\Bbbk$.
\item\label{lem4.2} 
If $\lambda\neq 0$ is atypical, then $\mathrm{Res}\, L(V)\cong 
L^{\mathtt{g}}(\lambda)$.
\item\label{lem4.3} 
If $\lambda\neq 0$ is typical and $\lambda_1-\lambda_2=1$, 
then $\mathrm{Res}\, L(V)\cong 
L^{\mathtt{g}}(\lambda)$.
\item\label{lem4.4} 
If $\lambda\neq 0$ is typical, dominant
and $\lambda_1-\lambda_2\neq 1$, then we have
$\mathrm{Res}\, L(V)\cong 
L^{\mathtt{g}}(\lambda)\oplus  L^{\mathtt{g}}(\lambda-\alpha)$.
\item\label{lem4.5} 
In all other cases we have $L(V)\cong M(V)$.
\end{enumerate}
\end{lemma}

\begin{proof}
See, for example, \cite{Pe2} or, alternatively,
\cite{Or} for full details.
\end{proof}

\subsection{Induction and coinduction}\label{s2.5} 

The restriction functor $\mathrm{Res}$ defined in the previous subsection
is obviously exact and hence admits
both a left and a right adjoint. As usual, the left adjoint of
$\mathrm{Res}$ is the induction functor 
\begin{displaymath}
\mathrm{Ind}=U(\mathfrak{q})\otimes_{U(\mathtt{g})}{}_- :
\mathtt{g}\text{-}\mathrm{Mod}\to \mathfrak{a}\text{-}\mathrm{sMod}.
\end{displaymath}
As usual, the right adjoint of $\mathrm{Res}$ is the coinduction functor,
however, by \cite[Proposition~22]{Fr}, it is isomorphic to
the induction functor $\mathrm{Ind}$. In particular, $\mathrm{Ind}$ is exact
(which also follows from the PBW theorem). From the PBW theorem it follows 
that the composition $\mathrm{Res}
\circ\mathrm{Ind}$ is isomorphic to the
endofunctor $\left(\displaystyle\bigwedge \mathfrak{q}_{\overline{1}}
\otimes_{\Bbbk}{}_-\right)_{\overline{0}}$ of 
$\mathtt{g}\text{-}\mathrm{Mod}$ (here 
$\mathfrak{q}_{\overline{1}}$ is considered as a purely odd 
$\mathtt{g}$-supermodule in the natural way).

\subsection{Localization of $U(\mathfrak{q})$}\label{s2.55} 

Consider the multiplicative subset $X=\{1,F,F^2,F^3,\dots\}$ of
$U(\mathfrak{q})$. The adjoint action of $F$ on $U(\mathfrak{q})$
is obviously locally nilpotent. Hence $X$ is an Ore subset of 
$U(\mathfrak{q})$ (see \cite[Lemma~4.2]{Mat}) and we can denote by $U'$ the
Ore localization of $U(\mathfrak{q})$ with respect to $X$. According to
\cite[Lemma~4.3]{Mat}, there exists a unique family $\theta_{z}$, 
$z\in\Bbbk$, of automorphisms of $U'$, which are polynomial in 
$z$ and satisfy the condition $\theta_{z}(u)=F^zuF^{-z}$,
$u\in U'$, provided that $z\in\mathbb{Z}$. 
From the PBW theorem we have that $F$ is not 
a zero divisor in $U(\mathfrak{q})$ and hence the canonical map
$U(\mathfrak{q})\to U'$ is injective. For $z\in\Bbbk$ denote by 
$\Theta_z$ the endofunctor of $\mathfrak{q}\text{-}\mathrm{sMod}$ defined
for $M\in \mathfrak{q}\text{-}\mathrm{sMod}$ as follows:
\begin{displaymath}
\Theta_z(M)=\left\{
v\in U'\otimes_{U(\mathfrak{q})} M \,:\,
E^k(v)=0\text{ for some }k\in\mathbb{N}
\right\},
\end{displaymath}
where the left action of $U(\mathfrak{q})$ on $U'$ is given by
multiplication with $\theta_{z}(u)$, $u\in U(\mathfrak{q})$. 
The functor $\Theta_z$ is defined on morphisms in the natural way.

\subsection{Primitive ideals}\label{s2.6} 

An important rough invariant of a simple $\mathfrak{q}$-supermodule
$L$ is its annihilator $\mathrm{Ann}_{U(\mathfrak{q})}(L)$, which 
is a graded primitive ideal in $U(\mathfrak{q})$. Hence it is 
important to know all primitive ideals in $U(\mathfrak{q})$. 
For $\lambda\in\mathfrak{h}^*_{\overline{0}}$ we set
\begin{gather*}
\mathcal{I}_{\lambda}=\mathrm{Ann}_{U(\mathfrak{q})}
(L(\mathcal{V}(\lambda)))=\mathrm{Ann}_{U(\mathfrak{q})}
(L(\Pi\mathcal{V}(\lambda))),\\
\mathcal{J}_{\lambda}=\mathrm{Ann}_{U(\mathfrak{q})}
(L(\mathcal{V}(\lambda))_{\overline{0}})=\mathrm{Ann}_{U(\mathfrak{q})}
(L(\Pi\mathcal{V}(\lambda))_{\overline{1}}),\\
\mathcal{J}'_{\lambda}=\mathrm{Ann}_{U(\mathfrak{q})}
(L(\mathcal{V}(\lambda))_{\overline{1}})=\mathrm{Ann}_{U(\mathfrak{q})}
(L(\Pi\mathcal{V}(\lambda))_{\overline{0}}).
\end{gather*}
By definition, $\mathcal{I}_{\lambda}$ is an ideal of $U(\mathfrak{q})$,
while $\mathcal{J}_{\lambda}$ and $\mathcal{J}'_{\lambda}$ are  
$U(\mathfrak{q})\text{-}U(\mathtt{g})$-sub\-bi\-mo\-du\-les
of $U(\mathfrak{q})$. Obviously, we have 
$\mathcal{I}_{\lambda}\subset \mathcal{J}_{\lambda}$,
$\mathcal{I}_{\lambda}\subset \mathcal{J}'_{\lambda}$, and
$\mathcal{I}_{\lambda}=\mathcal{J}_{\lambda}\cap \mathcal{J}'_{\lambda}$.

\begin{lemma}[\cite{Mu}]\label{lem5}
The map $\lambda\mapsto \mathcal{I}_{\lambda}$
is a surjection from $\mathfrak{h}^*_{\overline{0}}$ onto the set
of graded primitive ideals of $U(\mathfrak{q})$.
\end{lemma}

We would need to know the fibers of the map from Lemma~\ref{lem5}.
They are given by the following:

\begin{proposition}\label{prop5-new}
Let $\lambda,\mu\in \mathfrak{h}^*_{\overline{0}}$. Then the equality
$\mathcal{I}_{\lambda}=\mathcal{I}_{\mu}$ holds only in the case when
$\lambda=\mu$ or in the following cases:
\begin{enumerate}[(a)]
\item\label{prop5-new.1} $\lambda_1+\lambda_2\neq 0$,
$\lambda$ is not integral and $(\mu_1,\mu_2)=(\lambda_2,\lambda_1)$; 
\item\label{prop5-new.2}  $\lambda_1+\lambda_2=0$,
$\lambda$ is not integral and $(\mu_1,\mu_2)=(\lambda_2-1,\lambda_1+1)$. 
\end{enumerate}
\end{proposition}

\begin{proof}
The element $H_1+H_2$ acts on $L(\mathcal{V}(\lambda))$ via the scalar
$\lambda_1+\lambda_2$ and on $L(\mathcal{V}(\mu))$ via the scalar
$\mu_1+\mu_2$. Therefore $H_1+H_2-\lambda_1-\lambda_2\in 
\mathcal{I}_{\lambda}$ and $H_1+H_2-\mu_1-\mu_2\in 
\mathcal{I}_{\mu}$. Hence the equality 
$\mathcal{I}_{\lambda}=\mathcal{I}_{\mu}$ 
implies $\lambda_1+\lambda_2=\mu_1+\mu_2$.

Assume now that $\lambda_1+\lambda_2=\mu_1+\mu_2\neq 0$. Consider the
quadratic Casimir element $\mathtt{c}=(H_1-H_2+1)^2+4FE$. For
$\nu\in\mathfrak{h}^*_{\overline{0}}$ we have 
$(\mathtt{c}-(\nu_1-\nu_2+1)^2)L^{\mathtt{g}}(\nu)=0$
(see e.g. \cite[Chapter~3]{Ma}). Hence from Lemmata~\ref{lem4} 
and \ref{lempar} it follows that the ideal $\mathcal{I}_{\mu}$ contains the 
element $(\mathtt{c}-(\nu_1-\nu_2+1)^2)(\mathtt{c}-(\nu_1-\nu_2-1)^2)$.
This means that the equality
$\mathcal{I}_{\lambda}=\mathcal{I}_{\mu}$ 
implies $(\mu_1,\mu_2)=(\lambda_2,\lambda_1)$ provided that 
$\lambda\neq \mu$. 

Assume that $(\mu_1,\mu_2)=(\lambda_2,\lambda_1)$, $\lambda\neq \mu$, and
both $\lambda$ and $\mu$ are integral. Then without loss of generality we
may assume that $\lambda$ is dominant and regular. In this case 
$U(\mathfrak{q})/\mathcal{I}_{\lambda}$ is finite dimensional while 
$U(\mathfrak{q})/\mathcal{I}_{\mu}$ is not (by Lemma~\ref{lem2}).
Therefore the equality $\mathcal{I}_{\lambda}=\mathcal{I}_{\mu}$
implies that $\lambda$ is not integral.

Assume now that $\lambda_1+\lambda_2=\mu_1+\mu_2= 0$. Then from 
Lemmata~\ref{lem4} and \ref{lempar} it follows that 
$\mathtt{c}-(\lambda_1-\lambda_2+1)^2$ and
$\mathtt{c}-(\mu_1-\mu_2+1)^2$ belong to $\mathcal{I}_{\lambda}$ and 
$\mathcal{I}_{\mu}$, respectively. Hence the equality 
$\mathcal{I}_{\lambda}=\mathcal{I}_{\mu}$  implies the equality
$(\mu_1,\mu_2)=(\lambda_2-1,\lambda_1+1)$. Similarly to the
previous paragraph one also show that the equality
$\mathcal{I}_{\lambda}=\mathcal{I}_{\mu}$ implies that $\lambda$ is not 
integral. This proves necessity of the conditions \eqref{prop5-new.1} 
and \eqref{prop5-new.2}.

Let us now prove sufficiency of the condition \eqref{prop5-new.2}.
Assume that the weight $\lambda\in \mathfrak{h}^*_{\overline{0}}$ is 
atypical and not integral. Set $z=-2(\lambda_1-\lambda_2)-2$ and
$(\mu_1,\mu_2)=(\lambda_2-1,\lambda_1+1)$. 
Consider the $U(\mathfrak{q})$-supermodule
$M=\Theta_{z}(L(\mathcal{V}(\lambda)))$. Then a direct calculation 
(see e.g. \cite[Section~3.5]{Ma}) implies that 
$\mathrm{Res}\, M\cong L^{\mathtt{g}}(\mu)$ 
and hence we have either $M\cong L(\mathcal{V}(\mu))$ or
$M\cong \Pi L(\mathcal{V}(\mu))$. In particular, $\mathcal{I}_{\mu}$
coincides with the annihilator of the supermodule $M$.

Let $N=U'\otimes_{U(\mathfrak{q})}L(\mathcal{V}(\lambda))$ be the
usual induced $U'$-supermodule. Since $U'$ is an Ore localization of
$U$, we have $\mathrm{Ann}_{U'}N=U'\mathcal{I}_{\lambda}U'=:I$. As
$F$ acts injectively on $L(\mathcal{V}(\lambda))$ (since $\lambda$ is
not integral, see Lemma~\ref{lem4}), the supermodule $N$ contains $L(\mathcal{V}(\lambda))$ as a subspace and hence is not trivial.
Thus $I$ is a proper ideal of $U'$, which is obviously stable with
respect to all inner automorphisms of $U'$ (i.e. 
$\theta_z(I)= F^zIF^{-z}\subset I$ for all $z\in\mathbb{Z}$). 

Now let $u\in I$. The polynomiality of $\theta_z$, $z\in\Bbbk$, 
says that there exist $k\in\mathbb{N}$ and $a_0,\dots,a_k\in U'$
such that for all $z\in\Bbbk$
\begin{displaymath}
\theta_z(u)=\sum_{i=0}^k z^i a_i. 
\end{displaymath}
Set $b_j:=\theta_j(u)$, $j=0,1,\dots,k$. Then $b_j\in I$ by the previous 
paragraph. Inverting the nondegenerate
Vandermonde matrix we can express all $a_i$'s as linear combinations of
$b_j$'s showing that all $a_i\in I$. But then $\theta_z(u)\in I$
for all $z\in\Bbbk$. Thus $\theta_z(I)\subset I$ for all 
$z\in \Bbbk$. Hence if we twist the $U'$-action on
$N$ by $\theta_z$ (as in the definition of the functor $\Theta_z$),
the ideal $I$ will still annihilate the resulting $U'$-supermodule. 
This yields the inclusion $\mathcal{I}_{\lambda}\subset \mathcal{I}_{\mu}$.

Because of the symmetry of $\mu$ and $\lambda$ we obtain 
$\mathcal{I}_{\lambda}=\mathcal{I}_{\mu}$. This proves sufficiency 
of the condition \eqref{prop5-new.2}. Sufficiency of the 
condition \eqref{prop5-new.1} is proved similarly. This completes the proof.
\end{proof}

For $\lambda\in \mathfrak{h}^*_{\overline{0}}$ we denote by
$\mathrm{Irr}_{\lambda}$ the set of classes of simple 
$U(\mathfrak{q})$-supermodules with annihilator $\mathcal{I}_{\lambda}$
up to isomorphism and parity change.
We also denote by $\mathtt{I}_{\lambda}$ the annihilator (in 
$U(\mathtt{g})$) of the module $L^{\mathtt{g}}(\lambda)$ and by
$\mathtt{Irr}^{\mathtt{g}}_{\lambda}$ the set of isomorphism classes of simple
$U(\mathtt{g})$-modules with annihilator $\mathtt{I}_{\lambda}$.
We refer the reader to \cite{Bl}, \cite{Ba} or \cite[Chapter~6]{Ma}
for descriptions of $\mathtt{Irr}^{\mathtt{g}}_{\lambda}$.

\subsection{Category $\mathcal{O}$}\label{s2.65} 

Denote by $\mathcal{O}$ the full subcategory of 
$\mathfrak{q}\text{-}\mathrm{sMod}$ consisting of all finitely generated
weight supermodules, the action of $E$ on which is locally nilpotent.
Denote also by $\mathtt{O}$ the corresponding category of
$\mathtt{g}$-modules. As $U(\mathfrak{q})$ is a finite extension of
$\mathtt{g}$ it follows that $\mathrm{Res}\, \mathcal{O}\subset 
\mathtt{O}$. It then follows that $\mathrm{Ind}\, \mathtt{O}\subset 
\mathcal{O}$. As every object in $\mathtt{O}$ has finite length
(see \cite[Chapter~5]{Ma}), every object in $\mathcal{O}$ has
finite length as well.

For $\lambda\in \mathfrak{h}^*_{\overline{0}}$ we denote by
$\mathcal{O}_{\lambda}$ and $\mathtt{O}_{\lambda}$ the full subcategories 
of $\mathcal{O}$ and $\mathtt{O}$, respectively, consisting of all 
(super)modules $M$ satisfying the condition that $M_{\mu}\neq 0$,
$\mu\in \mathfrak{h}^*_{\overline{0}}$, implies $(\lambda_1-\lambda_2)-
(\mu_1-\mu_2)\in\mathbb{Z}$. The categories $\mathcal{O}$ and $\mathtt{O}$ 
then decompose into a direct sum of $\mathcal{O}_{\lambda}$'s 
and $\mathtt{O}_{\lambda}$'s, respectively.

\subsection{Harish-Chandra bimodules}\label{s2.7} 

Let $\mathfrak{q}\text{-}\mathrm{Mod}\text{-}\mathtt{g}$ denote the
category of all $U(\mathfrak{q})\text{-}U(\mathtt{g})$-(super)bimodules.
A $U(\mathfrak{q})\text{-}U(\mathtt{g})$-bimodule is called a 
{\em Harish-Chandra} bimodule if it is finitely generated and decomposes
into a direct sum of simple finite dimensional $\mathtt{g}$-modules
(each occurring with finite multiplicity) with respect to the adjoint 
action of $\mathtt{g}$. We denote by $\mathcal{H}$ the full subcategory
of $\mathfrak{q}\text{-}\mathrm{Mod}\text{-}\mathtt{g}$ consisting of
all Harish-Chandra bimodules. As any Harish-Chandra 
$U(\mathfrak{q})\text{-}U(\mathtt{g})$-bimodule is a Harish-Chandra 
$U(\mathtt{g})\text{-}U(\mathtt{g})$-bimodule by restriction,
it has finite length (even as 
$U(\mathtt{g})\text{-}U(\mathtt{g})$-bimodule), see \cite[5.7]{BG}.

A typical way to produce Harish-Chandra bimodules is the following:
Let $X$ be a $\mathtt{g}$-module and $Y$ be a $\mathfrak{q}$-supermodule.
Then $\mathrm{Hom}_{\Bbbk}(X,Y)$ has a natural structure of a
$U(\mathfrak{q})\text{-}U(\mathtt{g})$-(super)bimodule. Denote by 
$\mathcal{L}(X,Y)$ the subspace of $\mathrm{Hom}_{\Bbbk}(X,Y)$,
consisting of all elements, the adjoint action of $U(\mathtt{g})$
on which is locally finite. As $U(\mathfrak{q})$ is a finite extension
of $U(\mathtt{g})$, from \cite[Kapitel~6]{Ja} it follows that 
$\mathcal{L}(X,Y)$ is in fact a 
$U(\mathfrak{q})\text{-}U(\mathtt{g})$-subbimodule of 
$\mathrm{Hom}_{\Bbbk}(X,Y)$. In particular, $\mathcal{L}(X,Y)$
is a $U(\mathtt{g})$-bimodule by restriction. Under some natural 
conditions (for examples when both $X$ and $Y$ are simple (super)modules) 
one easily verifies that $\mathcal{L}(X,Y)$ is a Harish-Chandra bimodule.
For $\lambda\in \mathfrak{h}^*_{\overline{0}}$ denote by 
$\mathcal{H}_{\lambda}^{1}$ the full subcategory of 
$\mathcal{H}$, which consists of all bimodules $B$ satisfying
$B(\mathtt{c}-(\lambda_1-\lambda_2+1)^2)=0$.
Obviously, $\mathcal{H}_{\lambda}^{1}$ is an abelian category with finite
direct sums. We will need the following generalization
of \cite[Theorem~5.9]{BG}:

\begin{proposition}\label{propHC}
Let $\lambda\in \mathfrak{h}^*_{\overline{0}}$ be such that 
$\lambda_1-\lambda_2\not\in\{-2,-3,-4,\dots\}$. Then the functors
\begin{displaymath}
\xymatrix{
\mathcal{H}_{\lambda}^{1}\ar@/^/[rrrr]^{\mathrm{F}:=
{}_-\otimes_{U(\mathtt{g})}
M^{\mathtt{g}}(\lambda)}&&&&
\mathcal{O}_{\lambda}
\ar@/^/[llll]^{\mathrm{G}:=\mathcal{L}(M^{\mathtt{g}}(\lambda),{}_-)}
}
\end{displaymath}
are mutually inverse equivalences of categories.
\end{proposition}

\begin{proof}
From \cite[6.22]{Ja} it follows that $(\mathrm{F},\mathrm{G})$ is
an adjoint pair of functors. As every object in both
$\mathcal{H}_{\lambda}^{1}$ and $\mathcal{O}_{\lambda}$ has finite
length (see above) by standard arguments (see e.g. proof
of \cite[Theorem~3.7.3]{Ma}) it follows that it is enough to check 
that both functors $\mathrm{F}$ and $\mathrm{G}$ send simple objects
to simple objects.

Under our assumptions we have that $M^{\mathtt{g}}(\lambda)$ is
projective in $\mathtt{O}$ (see \cite[Chapter~5]{Ma}), in particular,
on the level of $\mathtt{g}$ modules the functors 
$\mathcal{L}(M^{\mathtt{g}}(\lambda),{}_-)$ and
${}_-\otimes_{U(\mathtt{g})} M^{\mathtt{g}}(\lambda)$ are
mutually inverse equivalences of categories by \cite[Theorem~5.9]{BG}.
This means, in particular, that if $L\in \mathcal{O}_{\lambda}$
is simple, then the character of the supermodule $\mathrm{F}\mathrm{G}\,L$
coincides with the character of $L$. From the natural transformation
$\mathrm{F}\mathrm{G}\to\mathrm{Id}_{\mathcal{O}_{\lambda}}$ 
(given by adjunction) we 
thus obtain that $\mathrm{F}\mathrm{G}\,L\cong L$. Similarly one
shows that $\mathrm{G}\mathrm{F}$ sends simple objects from
$\mathcal{H}_{\lambda}^{1}$ to simple objects in $\mathcal{O}_{\lambda}$. 
It follows that both $\mathrm{F}$ and $\mathrm{G}$ send simple objects 
to simple objects, which implies the claim of the proposition.
\end{proof}

From Proposition~\ref{propHC} we immediately get:

\begin{corollary}\label{corHC}
Let $\lambda,\mu\in \mathfrak{h}^*_{\overline{0}}$ be such that 
$\mu_1-\mu_2\not\in\{-2,-3,\dots\}$ and
$B$ be a simple Harish-Chandra bimodule satisfying the conditions
$\mathcal{I}_{\lambda}B=B(\mathtt{c}-(\mu_1-\mu_2+1)^2)=0$.
Then we have $(\lambda_1-\lambda_2)-(\mu_1-\mu_2)\in\mathbb{Z}$ and 
$B\cong\mathcal{L}(M^{\mathtt{g}}(\mu),L(\mathcal{V}(\lambda)))$
or $B\cong\mathcal{L}(M^{\mathtt{g}}(\mu),L(\Pi\mathcal{V}(\lambda)))$.
\end{corollary}

For $\lambda\in \mathfrak{h}^*_{\overline{0}}$ define
\begin{displaymath}
\mathcal{L}_{\lambda}=\mathcal{L}(L^{\mathtt{g}}(\lambda),
L(\mathcal{V}(\lambda))),\quad
\mathcal{M}_{\lambda}=\mathcal{L}(L^{\mathtt{g}}(\lambda-\alpha),
L(\mathcal{V}(\lambda))) 
\end{displaymath}
and
\begin{displaymath}
\mathcal{L}'_{\lambda}=
\begin{cases}
\mathcal{L}(M^{\mathtt{g}}(\lambda),
L(\mathcal{V}(\lambda))), & \lambda_1-\lambda_2\not\in \{-2,-3,-4,\dots\};\\
\mathcal{L}(M^{\mathtt{g}}(-\lambda-2),
L(\mathcal{V}(\lambda))), & \text{otherwise}.
\end{cases}
\end{displaymath}

The following lemma is one of our most important technical tools.

\begin{lemma}\label{lem555}
Let $\lambda\in \mathfrak{h}^*_{\overline{0}}$.
\begin{enumerate}[(i)]
\item\label{lem555.1}
If $\lambda_1-\lambda_2\not\in\mathbb{Z}\setminus\{-1\}$, then 
$\mathcal{L}_{\lambda}\cong \mathcal{L}'_{\lambda}$.
\item\label{lem555.2}
If $\lambda_1-\lambda_2\in\{-2,-3,-4,\dots\}$, then
for any simple infinite di\-men\-si\-o\-nal $\mathtt{g}$-module $X$
there is a short exact sequence
\begin{equation}\label{eq3}
0\to \mathrm{Ker}\to
\mathcal{L}'_{\lambda}\otimes_{U(\mathtt{g})}X\to
\mathcal{L}_{\lambda}\otimes_{U(\mathtt{g})}X\to 0
\end{equation}
of $\mathfrak{q}$-supermodules, where $\mathrm{Ker}$ is finite dimensional.
\end{enumerate}
\end{lemma}

\begin{proof}
The statement \eqref{lem555.1} is trivial for in this case we have 
an isomorphism $M^{\mathtt{g}}(\lambda)\cong L^{\mathtt{g}}(\lambda)$ (see 
\cite[Chapter~3]{Ma}).

To prove \eqref{lem555.2} we assume that 
$\lambda_1-\lambda_2\in \{-2,-3,-4,\dots\}$.
Applying the left exact functor $\mathcal{L}({}_-,\mathcal{V}(\lambda))$ 
to the short exact sequence
\begin{equation}\label{eq301}
0\to L^{\mathtt{g}}(\lambda)\to M^{\mathtt{g}}(-\lambda-\alpha)\to
L^{\mathtt{g}}(-\lambda-\alpha) \to 0
\end{equation}
we obtain the following exact sequence:
\begin{displaymath}
0\to \mathcal{L}(L^{\mathtt{g}}(-\lambda-\alpha),\mathcal{V}(\lambda))\to
\mathcal{L}'_{\lambda}\to \mathcal{L}_{\lambda}.
\end{displaymath}
The socle of the $\mathtt{g}$-module 
$\mathrm{Res}\, \mathcal{V}(\lambda)$ does not contain
finite dimensional submodules because of our choice of $\lambda$
and Lemma~\ref{lem4}. As $L^{\mathtt{g}}(-\lambda-\alpha)$ is 
finite dimensional, for every  finite dimensional  $\mathtt{g}$-module 
$V$ we thus have
\begin{equation}\label{eq2}
\mathrm{Hom}_{\mathtt{g}}(L^{\mathtt{g}}(-\lambda-\alpha)\otimes_{\Bbbk}
V,\mathrm{Res}\, \mathcal{V}(\lambda))=0.
\end{equation}
Hence $\mathcal{L}(L^{\mathtt{g}}(-\lambda-\alpha),\mathcal{V}(\lambda))=0$
by \cite[6.8]{Ja}. This equality implies existence of a natural inclusion
$\mathcal{L}'_{\lambda}\hookrightarrow \mathcal{L}_{\lambda}$.
Let us denote by $B$ the cokernel of this inclusion.

Let $\mu\in \mathfrak{h}^*_{\overline{0}}$ be such that 
$L^{\mathtt{g}}(\mu)$ is finite dimensional. If $\mu_1+\mu_2\neq 0$, then
the action of the central element $H_1+H_2$ on the modules 
$L^{\mathtt{g}}(\lambda)$ and $M^{\mathtt{g}}(-\lambda-\alpha)$ on one side
and the module $V\otimes_{\Bbbk}\mathrm{Res}\, \mathcal{V}(\lambda)$ on
the other side do not agree and hence such $V$ occurs as a direct 
summand (with respect to the adjoint action of $\mathtt{g}$) in 
neither $\mathcal{L}'_{\lambda}$ nor $\mathcal{L}_{\lambda}$.

If $\mu_1+\mu_2=0$ and $\mu_1-\mu_2=m\in \mathbb{N}_0$ is big enough,
then the module $L^{\mathtt{g}}(\mu)\otimes_{\Bbbk} 
L^{\mathtt{g}}(-\lambda-\alpha)$ decomposes into a direct sum of 
$L^{\mathtt{g}}(\nu)$, where $\nu\in\{-\lambda-\alpha+\mu,
-\lambda-2\alpha+\mu,\dots,\mu+\lambda+\alpha\}$
(see e.g. \cite[Section~1.4]{Ma}). For all $\mu$ big enough the
central characters of all these $L^{\mathtt{g}}(\nu)$ are different 
from the central characters of all simple subquotients in 
$\mathrm{Res}\, \mathcal{V}(\lambda)$. Hence from \cite[6.8]{Ja}
it follows that $L^{\mathtt{g}}(\mu)$ occurs with the
same multiplicity in $\mathcal{L}'_{\lambda}$ and $\mathcal{L}_{\lambda}$ 
with respect to the adjoint action. This implies that the bimodule $B$
is finite dimensional.

Now we claim that 
\begin{equation}\label{eq4}
B\otimes_{U(\mathtt{g})}X=0. 
\end{equation}
Assume that this
is not the case. Since $X$ is simple and $B$ is finite dimensional, 
the module $B\otimes_{U(\mathtt{g})}X$ is holonomic and hence has 
finite length. Let $N$ be some simple quotient of 
$B\otimes_{U(\mathtt{g})}X$. Then,  by adjunction, we have
\begin{displaymath}
0\neq \mathrm{Hom}_{\mathfrak{q}}(B\otimes_{U(\mathtt{g})}X,N)= 
\mathrm{Hom}_{\mathtt{g}}(X,\mathrm{Hom}_{\mathfrak{q}}(B,N)).
\end{displaymath}
If $N$ is finite dimensional, then $\mathrm{Hom}_{\mathfrak{q}}(B,N)$
is finite dimensional and thus
$\mathrm{Hom}_{\mathtt{g}}(X,\mathrm{Hom}_{\mathfrak{q}}(B,N))=0$,
a contradiction. If $N$ is infinite dimensional, then
$\mathrm{Hom}_{\mathfrak{q}}(B,N)=0$ as $B$ is finite dimensional,
which is again a contradiction. 

Applying ${}_-\otimes_{U(\mathtt{g})}X$ to the following
short exact sequence of bimodules
\begin{displaymath}
0\to\mathcal{L}'_{\lambda}\to \mathcal{L}_{\lambda}\to B\to 0
\end{displaymath}
and using \eqref{eq4} and the right exactness 
of the tensor product we obtain an exact sequence 
\begin{displaymath}
\mathrm{Tor}_1^{U(\mathtt{g})}(B,X)\to
\mathcal{L}'_{\lambda}\otimes_{U(\mathtt{g})}X
\to \mathcal{L}_{\lambda}\otimes_{U(\mathtt{g})}X \to 0.
\end{displaymath}

Now we just recall that $U(\mathtt{g})$ has finite global dimension. 
So we can take a minimal finite free resolution of $X$ (where each component
will be finitely generated as $U(\mathtt{g})$ is noetherian), tensor 
it with our finite dimensional bimodule $B$ and obtain a finite complex 
of finite dimensional vector spaces. All torsion groups from $B$ to $X$
are homologies of this complex, hence finite dimensional.
The claim \eqref{lem555.2} follows, which completes the proof of 
the proposition.
\end{proof}

\section{Classification of simple $\mathfrak{q}_{2}$-supermodules}\label{s3} 

In this section we prove Theorem~\ref{mthm1} and some related results.
All this is divided into separate steps, organized as subsections.

\subsection{All simple $\mathfrak{q}$-supermodules are subsupermodules 
of some induced supermodules}\label{s3.1} 

We start with the following observation:

\begin{proposition}\label{prop7}
Let $N$ be a simple $\mathfrak{q}$-supermodule. Then there exists 
a simple $\mathtt{g}$-module $L$ such that either $N$ or
$\Pi N$ is a subsupermodule of the supermodule $\mathrm{Ind}\, L$.
\end{proposition}

\begin{proof}
The supermodule $N$ is simple, in particular, it is finitely generated.
As $U(\mathfrak{q})$ is a finite extension of $U(\mathtt{g})$, it 
follows that the $\mathtt{g}$-module 
$\mathrm{Res}_{\mathtt{g}}^{\mathfrak{q}}\, N$ is 
finitely generated as well. Let $\{v_1,\dots,v_k\}$ be some minimal
generating system of $\mathrm{Res}_{\mathtt{g}}^{\mathfrak{q}}\, N$ and 
\begin{displaymath}
M=\mathrm{Res}_{\mathtt{g}}^{\mathfrak{q}}\, 
N/(U(\mathtt{g})\{v_2,\dots,v_k\}).
\end{displaymath}
then the module $M$ is non-zero because of the minimality of
the system $\{v_1,\dots,v_k\}$. Moreover, $M$ is generated by one
element $v$ (the image of $v_1$ in $M$). Let $I=\{u\in
U(\mathtt{g}): u(v)=0\}$. Then $I$ is a left ideal of $U(\mathtt{g})$
different from $U(\mathtt{g})$ and $M\cong U(\mathtt{g})/I$. 
Let $J$ be some maximal left ideal of $U(\mathtt{g})$, 
containing $I$ (such $J$ exists by Zorn's
lemma). Consider the module 
$L=M/JM$, which is a simple quotient of $M$ by construction.
Changing, if necessary, the parity of $N$, we may assume that $L$ is even.
As $M$ is a quotient of $\mathrm{Res}_{\mathtt{g}}^{\mathfrak{q}}\, N$, 
we obtain that $L$ is a simple quotient of either
$\mathrm{Res}\, N$ or $\mathrm{Res}\, \Pi N$. We consider the
first case and the second one is dealt with by similar arguments.

Using the adjunction between $\mathrm{Res}$ and $\mathrm{Ind}$ (see 
Subsection~\ref{s2.5}) we have
\begin{displaymath}
0\neq \mathrm{Hom}_{\mathtt{g}}(\mathrm{Res}\, N, L)=
\mathrm{Hom}_{\mathfrak{q}}(N,\mathrm{Ind}\, L).
\end{displaymath}
The claim of the proposition follows.
\end{proof}

\subsection{Finite length of the restriction}\label{s3.2} 

Now we are ready to prove the first part of Theorem~\ref{mthm1}:

\begin{proposition}\label{prop8}
Let $N$ be a simple $\mathfrak{q}$-supermodule. Then the 
$\mathtt{g}$-module $\mathrm{Res}_{\mathtt{g}}^{\mathfrak{q}}\, N$ 
has finite length.
\end{proposition}

\begin{proof}
By Proposition~\ref{prop7}, there exists a simple $\mathtt{g}$-module 
$L$ such that either $N$ or $\Pi N$ is a subsupermodule of the 
supermodule $\mathrm{Ind}\, L$. Changing, if necessary, the parity of $N$
we may assume that $N$ is a subsupermodule of $\mathrm{Ind}\, L$.
Hence $\mathrm{Res}_{\mathtt{g}}^{\mathfrak{q}}\, N$ 
is a submodule of the $\mathtt{g}$-module
$\mathrm{Res}_{\mathtt{g}}^{\mathfrak{q}}\circ\mathrm{Ind}\, L$. 
The latter is isomorphic to the
$\mathtt{g}$-module $\bigwedge \mathfrak{q}_{\overline{1}}
\otimes_{\Bbbk} L$, see Subsection~\ref{s2.5}.

From the classification of all simple $\mathtt{g}$-modules (see
\cite{Bl,Ba} or \cite[Chapter~6]{Ma}) we obtain that the module
$L$ is holonomic in the sense of \cite{Ba3}. 
As tensoring with the finite dimensional module
$\bigwedge \mathfrak{q}_{\overline{1}}$
cannot increase the Gelfand-Kirillov dimension of the module $L$, 
it follows that the module $\bigwedge \mathfrak{q}_{\overline{1}}
\otimes_{\Bbbk} L$ is holonomic as well. Hence it has finite
length, see \cite[Section~3]{Ba3}. This implies that the submodule
$\mathrm{Res}_{\mathtt{g}}^{\mathfrak{q}}\, N$ has 
finite length as well. The proof is complete.
\end{proof}

\subsection{Finite-dimensional supermodules}\label{s3.3} 
From Lemma~\ref{lem5} and Proposition~\ref{prop5-new} we have a complete
description of all different primitive ideals in $U(\mathfrak{q})$.
We approach our classification via a case-by-case analysis and start
with the easiest case of finite dimensional supermodules.

\begin{lemma}\label{lem9}
Let $\lambda\in\mathfrak{h}^*_{\overline{0}}$ be either zero
or dominant. Then
$L(\mathcal{V}(\lambda))$ is a unique (up to isomorphism
and parity change) 
simple supermodule with annihilator $\mathcal{I}_{\lambda}$.
\end{lemma}

\begin{proof}
From Lemma~\ref{lem2} it follows that under our assumptions we
have $\dim U(\mathfrak{q})/\mathcal{I}_{\lambda}<\infty$. Therefore any
simple supermodule $L$ with annihilator $\mathcal{I}_{\lambda}$ must be
finite dimensional. Now the claim follows from Lemma~\ref{lem2}
and Proposition~\ref{prop5-new}.
\end{proof}

\subsection{Atypical supermodules}\label{s3.4} 

Surprisingly enough the second easiest case is that of atypical
supermodules. This is due to the easiest possible structure of atypical
simple highest weight supermodules (see Lemma~\ref{lem4}).
Note that the case of all (in particular, 
atypical) finite di\-men\-si\-o\-nal supermodules
has already been taken care of in Subsection~\ref{s3.3}. In the present
subsection we consider the case of atypical infinite dimensional 
supermodules and classify such supermodules via certain simple 
$U(\mathtt{g})$-modules. We also explicitly describe the restriction of 
each simple infinite dimensional atypical $U(\mathfrak{q})$-supermodule to 
$U(\mathtt{g})$.

\begin{proposition}\label{prop10}
Let $t\in \Bbbk\setminus \frac{1}{2}\mathbb{N}_0$ and $\lambda=(t,-t)$.
\begin{enumerate}[(i)]
\item\label{prop10.1} 
The following correspondence is a bijection between 
$\mathtt{Irr}^{\mathtt{g}}_{\lambda}$ and $\mathrm{Irr}_{\lambda}$:
\begin{displaymath}
\begin{array}{rcl}
\mathtt{Irr}^{\mathtt{g}}_{\lambda}&\leftrightarrow&\mathrm{Irr}_{\lambda}\\
L&\mapsto&\displaystyle
 \mathcal{L}_{\lambda}\otimes_{U(\mathtt{g})}L
\end{array}
\end{displaymath}
\item\label{prop10.2}
For any $L\in \mathtt{Irr}^{\mathtt{g}}_{\lambda}$ the module 
$\mathrm{Res}_{\mathtt{g}}^{\mathfrak{q}}\, 
\mathcal{L}_{\lambda}\otimes_{U(\mathtt{g})}L$ 
is semi-simple and we have 
$\mathrm{Res}_{\mathtt{g}}^{\mathfrak{q}}\, 
\mathcal{L}_{\lambda}\otimes_{U(\mathtt{g})}L\cong
L\oplus L$ and $\mathrm{Res}\,\mathcal{L}_{\lambda}
\otimes_{U(\mathtt{g})}L\cong L$.
\end{enumerate}
\end{proposition}

\begin{proof}
We start with the statement \eqref{prop10.2}. We observe that 
$L^{\mathtt{g}}(\lambda)=M^{\mathtt{g}}(\lambda)$ because of our
restrictions on $\lambda$ (see e.g. \cite[Chapter~3]{Ma}). Consider 
$\mathcal{L}_{\lambda}$ as a $U(\mathtt{g})$-bimodule. We have
\begin{equation}\label{eq1}
\begin{array}{rcl}
\mathcal{L}_{\lambda}&:=& \mathcal{L}(L^{\mathtt{g}}(\lambda),
L(\mathcal{V}(\lambda)))\\
\text{(by Lemmata~\ref{lem4} and \ref{lempar})}
&\cong& \mathcal{L}(L^{\mathtt{g}}(\lambda),
L^{\mathtt{g}}(\lambda)\oplus L^{\mathtt{g}}(\lambda))\\
\text{(by \cite[6.8]{Ja})}&\cong& \mathcal{L}(L^{\mathtt{g}}(\lambda),
L^{\mathtt{g}}(\lambda))\oplus\mathcal{L}(L^{\mathtt{g}}(\lambda),
L^{\mathtt{g}}(\lambda))\\
\text{(by \cite[7.25]{Ja})}&\cong& U(\mathtt{g})/\mathtt{I}_{\lambda}\oplus 
U(\mathtt{g})/\mathtt{I}_{\lambda}.
\end{array}
\end{equation}
As for any $L\in \mathtt{Irr}^{\mathtt{g}}_{\lambda}$ we have 
$\mathtt{I}_{\lambda}L=0$, we deduce that
\begin{displaymath}
\mathrm{Res}_{\mathtt{g}}^{\mathfrak{q}}\, 
\mathcal{L}_{\lambda}\otimes_{U(\mathtt{g})}L\cong
(U(\mathtt{g})/\mathtt{I}_{\lambda}\oplus 
U(\mathtt{g})/\mathtt{I}_{\lambda})\otimes_{U(\mathtt{g})}L\cong
L\oplus L,
\end{displaymath}
proving \eqref{prop10.2}.

To prove \eqref{prop10.1} we first take some 
$L\in \mathtt{Irr}^{\mathtt{g}}_{\lambda}$ and consider the supermodule
$N=\mathcal{L}_{\lambda}\otimes_{U(\mathtt{g})}L$. Then from the
definition of $\mathcal{L}_{\lambda}$ we obtain 
$\mathcal{I}_{\lambda}N=0$. If $N$ is not simple, then from 
\eqref{prop10.2} it follows that either $N_{\overline{0}}$
or $N_{\overline{1}}$ is a simple submodule of $N$. Assume that 
$N_{\overline{0}}$ is a submodule (for $N_{\overline{1}}$ the arguments
are similar). Then $U(\mathfrak{q})N_{\overline{0}}\subset
N_{\overline{0}}$, which implies that $N_{\overline{0}}$ is annihilated
by $U(\mathfrak{q})_{\overline{1}}$. 

If $\mu\in\mathfrak{h}^*_{\overline{0}}$ is nonzero, then from 
Subsection~\ref{s2.3} it follows that at least one of the elements 
$\overline{H}_1$ or $\overline{H}_2$ does not  annihilate $L(\mathcal{V}(\lambda))$. Therefore $\mathcal{I}_{0}$ 
is the only primitive ideal of $U(\mathfrak{q})$, which 
contains $U(\mathfrak{q})_{\overline{1}}$.
From Lemma~\ref{lem9} it thus follows that $N_{\overline{0}}$ 
must be the trivial $U(\mathfrak{q})$-supermodule, which is impossible as
$\lambda\neq 0$. This contradiction shows that the supermodule $N$ is 
simple and hence $N\in \mathrm{Irr}_{\lambda}$. In particular, the
correspondence from \eqref{prop10.1} is well-defined. From 
\eqref{prop10.2} it follows immediately that it is even injective.

So, to complete the proof we have to show that every $N\in 
\mathrm{Irr}_{\lambda}$ has the form 
$\mathcal{L}_{\lambda}\otimes_{U(\mathtt{g})}L$ for some 
$L\in \mathtt{Irr}^{\mathtt{g}}_{\lambda}$.  By Proposition~\ref{prop8},
the $U(\mathtt{g})$-module $\mathrm{Res}\, N$ has finite length.
Let $L$ be a simple submodule of this module $\mathrm{Res}\, N$. 

We have $\mathcal{I}_{\lambda}L=0$. From Lemma~\ref{lem4} 
and \cite[Chapter~3]{Ma} it follows that 
$\mathtt{c}-(\lambda_1-\lambda_2+1)^2\in \mathcal{I}_{\lambda}$ and 
thus $\mathtt{c}-(\lambda_1-\lambda_2+1)^2$
annihilates $L$, which yields that $\mathtt{I}_{\lambda}L=0$.

Now we claim that $L$ cannot be finite dimensional. Indeed, if 
$L$ would be finite dimensional, using the adjunction between
$\mathrm{Ind}$ and $\mathrm{Res}$ we would have
\begin{displaymath}
0\neq \mathrm{Hom}_{\mathtt{g}}(L,\mathrm{Res}\, N)=
\mathrm{Hom}_{\mathfrak{q}}(\mathrm{Ind}\, L,N).
\end{displaymath}
The supermodule $\mathrm{Ind}\, L$, which is isomorphic to the module
$\bigwedge \mathfrak{q}_{\overline{1}} \otimes_{\Bbbk}L$ as
$\mathtt{g}$-module, would thus be finite dimensional,
which would imply that $N$ is finite dimensional as well. This
contradicts Lemma~\ref{lem9} and thus $L$ is infinite dimensional.
This implies $L\in\mathtt{Irr}^{\mathtt{g}}_{\lambda}$ (see \cite[Chapter~3]{Ma}).

Consider the bimodule $\mathcal{L}(L,N)$, which is a Harish-Chandra
bimodule by \cite[6.8]{Ja} and \cite[5.7]{BG}. The inclusion 
$L\hookrightarrow N$ is a $\mathtt{g}$-homomorphism and hence
is annihilated by the adjoint action of $\mathtt{g}$. Hence this inclusion
is a nontrivial element of $\mathcal{L}(L,N)$, which shows that 
$\mathcal{L}(L,N)\neq 0$. 

From the above and $\mathcal{I}_{\lambda}N=0$ we have  
$\mathcal{I}_{\lambda}\mathcal{L}(L,N)=
\mathcal{L}(L,N)\mathtt{I}_{\lambda}=0$. As $\mathcal{L}(L,N)$
has finite length, it thus must contain some simple Harish-Chandra
subbimodule $B$ such that $\mathcal{I}_{\lambda}B=B\mathtt{I}_{\lambda}=0$. 
Changing, if necessary, the parity of $N$ and using Corollary~\ref{corHC} 
we get that $B\cong \mathcal{L}'_{\lambda}$.

This implies the following:
\begin{displaymath}
\begin{array}{rcl}
0&\neq &\mathrm{Hom}_{U(\mathfrak{q})\text{-}U(\mathtt{g})}
(\mathcal{L}'_{\lambda},\mathcal{L}(L,N))\\
\text{(by definition of $\mathcal{L}(L,N)$)}&
\subset  &\mathrm{Hom}_{U(\mathfrak{q})\text{-}U(\mathtt{g})}
(\mathcal{L}'_{\lambda},\mathrm{Hom}_{\Bbbk}(L,N))\\
\text{(by adjunction)}&
\cong  &\mathrm{Hom}_{\mathfrak{q}}
(\mathcal{L}'_{\lambda}\otimes_{U(\mathtt{g})}L,N).
\end{array}
\end{displaymath}
At the same time, applying $\mathrm{Hom}_{\mathfrak{q}}({}_-,N)$
to the short exact sequence \eqref{eq3} we obtain the exact sequence
\begin{displaymath}
0\to  \mathrm{Hom}_{\mathfrak{q}}
(\mathcal{L}_{\lambda}\otimes_{U(\mathtt{g})}L,N)
\to  \mathrm{Hom}_{\mathfrak{q}}
(\mathcal{L}'_{\lambda}\otimes_{U(\mathtt{g})}L,N)
\to  \mathrm{Hom}_{\mathfrak{q}}
(\mathrm{Ker},N).
\end{displaymath}
As $\mathrm{Ker}$ is finite dimensional, while $N$ is simple
infinite dimensional, we get $\mathrm{Hom}_{\mathfrak{q}}
(\mathrm{Ker},N)=0$ and hence 
\begin{displaymath}
\mathrm{Hom}_{\mathfrak{q}}
(\mathcal{L}_{\lambda}\otimes_{U(\mathtt{g})}L,N)
\cong \mathrm{Hom}_{\mathfrak{q}}
(\mathcal{L}'_{\lambda}\otimes_{U(\mathtt{g})}L,N)\neq 0.
\end{displaymath}
As we already know that the supermodule
$\mathcal{L}_{\lambda}\otimes_{U(\mathtt{g})} L$ is simple (see the 
first part of the proof above), 
we conclude that  $\mathcal{L}_{\lambda}\otimes_{U(\mathtt{g})} L\cong N$
using Schur's lemma. This completes the proof of the claim \eqref{prop10.1} 
and of the whole proposition.
\end{proof}

After Proposition~\ref{prop10} it is natural to ask whether
the bimodule $\mathcal{L}_{\lambda}$ can be described explicitly. 
We will actually need this description later on.
This is done in the following:

\begin{lemma}\label{lem11}
Let $\lambda$ be as  in Proposition~\ref{prop10}.
The $U(\mathfrak{q})\text{-}U(\mathtt{g})$-bimodule 
$\mathcal{L}_{\lambda}$ is isomorphic (as a bimodule) 
either to the  $U(\mathfrak{q})\text{-}U(\mathtt{g})$-bimodule
$U(\mathfrak{q})/\mathcal{J}_{\lambda}$ or to the  
$U(\mathfrak{q})\text{-}U(\mathtt{g})$-bimodule
$U(\mathfrak{q})/\mathcal{J}'_{\lambda}$ 
(in which the right $U(\mathtt{g})$-struc\-ture is 
naturally given by the right multiplication).
\end{lemma}

\begin{proof}
We have either $\sqrt{t}-\mathbf{i}\sqrt{-t}\neq 0$ or 
$\sqrt{t}+\mathbf{i}\sqrt{-t}\neq 0$
(for otherwise $t=0$, which contradicts our choice of $\lambda$). We
consider the first case $\sqrt{t}-\mathbf{i}\sqrt{-t}\neq 0$.

By Lemma~\ref{lem4} there exists a $\mathtt{g}$-monomorphism
$\varphi:L^{\mathtt{g}}(\lambda)\to L(\mathcal{V}(\lambda))_{\overline{0}}$.
As $\varphi$ commutes with all elements from 
$\mathtt{g}$, the adjoint action of $\mathtt{g}$ on $\varphi$ 
is zero and hence $\varphi\in \mathcal{L}_{\lambda}$. 
Observe that the $\overline{H}_1+\overline{H}_2$ commutes with all elements 
from $\mathfrak{g}$. As $\sqrt{t}-\mathbf{i}\sqrt{-t}\neq 0$, 
from Subsection~\ref{s2.3} it follows that multiplication with 
the element  $\overline{H}_1+\overline{H}_2$ defines a nonzero 
homomorphism from  $L(\mathcal{V}(\lambda))_{\overline{0}}$ 
to $L(\mathcal{V}(\lambda))_{\overline{1}}$. From Lemma~\ref{lem4} we 
even get that this homomorphism is an isomorphism. We have
\begin{displaymath}
\begin{array}{rcl}
\mathcal{L}_{\lambda} &\supset& U(\mathfrak{q}) \varphi \\
\text{(using grading)}&\supset&U(\mathtt{g})\varphi\oplus 
U(\mathtt{g})(\overline{H}_1+\overline{H}_2)\varphi\\
\text{(by above)}&=&U(\mathtt{g})\varphi\oplus 
(\overline{H}_1+\overline{H}_2)U(\mathtt{g})\varphi\\
\text{(by \cite[7.25]{Ja})}&\cong&
\mathcal{L}(L^{\mathtt{g}}(\lambda),L^{\mathtt{g}}(\lambda))\oplus 
\mathcal{L}(L^{\mathtt{g}}(\lambda),L^{\mathtt{g}}(\lambda))\\
\text{(by Lemma~\ref{lem4})}&\cong&\mathcal{L}_{\lambda}.
\end{array}
\end{displaymath}
Hence $U(\mathfrak{q}) \varphi\cong U(\mathfrak{q})/\mathcal{J}_{\lambda}
\cdot \varphi\cong \mathcal{L}_{\lambda}$ (note that
$\mathcal{J}_{\lambda}\varphi=0$ by definition of $\mathcal{J}_{\lambda}$).
Thus we obtain that the map 
\begin{displaymath}
\begin{array}{rcl}
U(\mathfrak{q})/\mathcal{J}_{\lambda}&\to & \mathcal{L}_{\lambda}\\
u+\mathcal{J}_{\lambda} &\mapsto & u\cdot \varphi
\end{array}
\end{displaymath}
is bijective. 
Since $\varphi$ is a  $\mathtt{g}$-homomorphism,
this map is a homomorphism of 
$U(\mathfrak{q})\text{-}U(\mathtt{g})$-bimodules. 

The case $\sqrt{t}+\mathbf{i}\sqrt{-t}\neq 0$ reduces to the case
$\sqrt{t}-\mathbf{i}\sqrt{-t}\neq 0$ changing parity (by Subsection~\ref{s2.3}) 
and hence leads to the appearance of the ideal $\mathcal{J}'_{\lambda}$
instead of $\mathcal{J}_{\lambda}$. This completes the proof.
\end{proof}

\begin{lemma}\label{lem111}
Let $\lambda$ be as  in Proposition~\ref{prop10}.
\begin{enumerate}[(i)]
\item\label{lem111.1} If $\sqrt{t}-\mathbf{i}\sqrt{-t}\neq 0$, then 
the ideal $\mathcal{J}_{\lambda}$ is generated (as a
$U(\mathfrak{q})\text{-}U(\mathtt{g})$-bimodule) by $\mathtt{I}_{\lambda}$
and the element $\overline{H}_1-\overline{H}_2$.
\item\label{lem111.2} If $\sqrt{t}+\mathbf{i}\sqrt{-t}\neq 0$, then 
the ideal $\mathcal{J}'_{\lambda}$ is generated (as a
$U(\mathfrak{q})\text{-}U(\mathtt{g})$-bimodule) by $\mathtt{I}_{\lambda}$
and the element $\overline{H}_1-\overline{H}_2$.
\end{enumerate}
\end{lemma}

\begin{proof}
The statement \eqref{lem111.2} reduces to \eqref{lem111.1} by parity
change, so we prove the statement \eqref{lem111.1}. First let us
show that $\overline{H}_1-\overline{H}_2$ belongs to 
$\mathcal{J}_{\lambda}$. 

If $\sqrt{t}-\mathbf{i}\sqrt{-t}\neq 0$, then 
$\overline{H}_1+\overline{H}_2$ induces an isomorphism from $L(\mathcal{V}(\lambda))_{\overline{0}}$ to
$L(\mathcal{V}(\lambda))_{\overline{1}}$ (see proof of Lemma~\ref{lem11}).
From Subsection~\ref{s2.3} it then follows that
$(\overline{H}_1-\overline{H}_2)L(\mathcal{V}(\lambda))_{\overline{0}}=0$
and hence $\overline{H}_1-\overline{H}_2\in \mathcal{J}_{\lambda}$.
Let $J$ denote the $U(\mathfrak{q})\text{-}U(\mathtt{g})$-subbimodule
of $\mathcal{J}_{\lambda}$, generated by $\mathtt{I}_{\lambda}$
and the element $\overline{H}_1-\overline{H}_2$.

Applying to $\overline{H}_1-\overline{H}_2$ the adjoint action of 
$\mathtt{g}$ we obtain $\overline{F},\overline{E}\in J$.
Multiplying these with elements from $U(\mathfrak{q})$ from the left
we get that $J$ also contains the following elements:
\begin{gather*}
\overline{H}_1\overline{H}_2 \overline{F}\overline{E},\,
\overline{H}_2 \overline{F}\overline{E},\,
\overline{H}_1\overline{F}\overline{E},\,
\overline{H}_1\overline{H}_2\overline{E},\,
\overline{H}_1\overline{H}_2 \overline{F},\\
\overline{H}_1\overline{E},\,\overline{H}_2\overline{E},\,
\overline{F}\overline{E},\,
\overline{H}_1\overline{F},\,\overline{H}_2\overline{F},\,
H_1-\overline{H}_1\overline{H}_2.
\end{gather*}
Now from the PBW theorem it follows that the quotient 
$U(\mathfrak{q})/J$, as a right $U(\mathtt{g})$-module, is
a quotient of $U(\mathtt{g})/\mathtt{I}_{\lambda}+
(\overline{H}_1+\overline{H}_2)U(\mathtt{g})/\mathtt{I}_{\lambda}$.
From Lemma~\ref{lem11} it thus follows that $J=\mathcal{J}_{\lambda}$.
\end{proof}

\subsection{Typical regular nonintegral supermodules}\label{s3.5} 

Now we move to the easiest typical case, that is the case of 
typical regular nonintegral $\lambda$. In this case all
Verma supermodules are irreducible and this substantially simplifies
arguments. For generic supermodules a stronger result can be deduced
from \cite{Pe}.

Set $\lambda'=\lambda-\alpha$. Let for the moment $V$ denote the
$3$-dimensional simple $\mathtt{g}$-module with the trivial action of
$H_1+H_2$. Let $\mathfrak{C}_{\lambda}$ and $\mathfrak{C}_{\lambda'}$
denote the full subcategories of $\mathtt{g}\text{-}\mathrm{Mod}$,
consisting of all modules, the action of the elements 
$\mathtt{c}-(\lambda_1-\lambda_2+1)^2$ and
$\mathtt{c}-(\lambda_1-\lambda_2-1)^2$, respectively, on which is
locally finite. Recall (see e.g. \cite[4.1]{BG}) that the translation 
functor
\begin{displaymath}
\mathrm{T}_{\lambda}^{\lambda'}=\mathrm{proj}_{\mathfrak{C}_{\lambda'}} 
\circ V\otimes_{\Bbbk}{}_-
:\mathfrak{C}_{\lambda}\to \mathfrak{C}_{\lambda'}
\end{displaymath}
is an equivalence of categories (here it is important that 
$\lambda$ is not integral). In particular, it sends simple modules
to simple modules. We also denote by $\mathrm{T}_{\lambda'}^{\lambda}$ 
the translation  functor from $\mathfrak{C}_{\lambda'}$ 
to $\mathfrak{C}_{\lambda}$ (which is the inverse of 
$\mathrm{T}_{\lambda}^{\lambda'}$).

\begin{proposition}\label{prop21}
Assume that $\lambda$ is typical, regular and nonintegral.
\begin{enumerate}[(i)]
\item\label{prop21.1}
The following correspondence is a bijection between 
$\mathtt{Irr}^{\mathtt{g}}_{\lambda}$ and $\mathrm{Irr}_{\lambda}$:
\begin{displaymath}
\begin{array}{rcl}
\mathtt{Irr}^{\mathtt{g}}_{\lambda}&\leftrightarrow&\mathrm{Irr}_{\lambda}\\
L&\mapsto&\displaystyle
 \mathcal{L}_{\lambda}\otimes_{U(\mathtt{g})}L
\end{array}
\end{displaymath}
\item\label{prop21.2}
For every simple $L\in \mathtt{Irr}^{\mathtt{g}}_{\lambda}$ the module 
$\mathrm{Res}_{\mathtt{g}}^{\mathfrak{q}}\, 
\mathcal{L}_{\lambda}\otimes_{U(\mathtt{g})}L$ is 
semi-simple and we have
\begin{displaymath}
\mathrm{Res}_{\mathtt{g}}^{\mathfrak{q}}\, 
\mathcal{L}_{\lambda}\otimes_{U(\mathtt{g})}L\cong
L\oplus \mathrm{T}_{\lambda}^{\lambda'}L \oplus L
\oplus \mathrm{T}_{\lambda}^{\lambda'}L,
\end{displaymath}
where $\mathrm{T}_{\lambda}^{\lambda'}L$ is a simple module. Moreover,
we also have the isomorphism 
$\mathrm{Res}\,\mathcal{L}_{\lambda}\otimes_{U(\mathtt{g})}L\cong
L\oplus \mathrm{T}_{\lambda}^{\lambda'}L$.
\end{enumerate}
\end{proposition}

\begin{proof}
We again prove the claim \eqref{prop21.2} first. From 
Lemmata~\ref{lem4} and \ref{lempar} we obtain that, 
after restriction of the left 
action to  $U(\mathtt{g})$, the 
$U(\mathtt{g})\text{-}U(\mathtt{g})$-bimodule 
$\mathcal{L}_{\lambda}$  decomposes as follows:
\begin{multline}\label{eq11}
\mathcal{L}_{\lambda}\cong 
\mathcal{L}(L^{\mathtt{g}}(\lambda),L^{\mathtt{g}}(\lambda))\oplus
\mathcal{L}(L^{\mathtt{g}}(\lambda),L^{\mathtt{g}}(\lambda))\oplus\\
\oplus
\mathcal{L}(L^{\mathtt{g}}(\lambda),L^{\mathtt{g}}(\lambda'))\oplus
\mathcal{L}(L^{\mathtt{g}}(\lambda),L^{\mathtt{g}}(\lambda')).
\end{multline}
The first two direct summands (one even and one odd) are isomorphic 
to  $U(\mathtt{g})/\mathtt{I}_{\lambda}$
and result into the direct summand $L\oplus L$ of
$\mathrm{Res}_{\mathtt{g}}^{\mathfrak{q}}\, \mathcal{L}_{\lambda}\otimes_{U(\mathtt{g})}L$
(the even one also gives the direct summand $L$ of $\mathrm{Res}\,L$)
similarly to the proof of Proposition~\ref{prop10}\eqref{prop10.2}.

As both $L^{\mathtt{g}}(\lambda)$ and $L^{\mathtt{g}}(\lambda')$
are projective in $\mathtt{O}$ because of our choice of $\lambda$,
from \cite[3.3]{BG} we derive that the functor 
$\mathcal{L}(L^{\mathtt{g}}(\lambda),L^{\mathtt{g}}(\lambda'))
\otimes_{U(\mathtt{g})}{}_-$ is a projective
functor isomorphic to $\mathrm{T}_{\lambda}^{\lambda'}$. The claim
\eqref{prop21.2} follows.

Now let us prove that for any $L\in \mathtt{Irr}^{\mathtt{g}}_{\lambda}$ the 
$U(\mathfrak{q})$-supermodule 
$\mathcal{L}_{\lambda}\otimes_{U(\mathtt{g})}L$ is simple. As was
already mentioned in the proof of Proposition~\ref{prop10}, the element
$\overline{H}_1+\overline{H}_2$ commutes with all element from 
$U(\mathtt{g})$. As $(\overline{H}_1+\overline{H}_2)^2=H_1+H_2$
and  our $\lambda$ is now typical, we deduce that
for any simple $U(\mathfrak{q})$-supermodule $N\in  
\mathrm{Irr}_{\lambda}$ the multiplication with 
$\overline{H}_1+\overline{H}_2$  gives an isomorphism from 
the $U(\mathtt{g})$-module $N_{\overline{0}}$ to the
$U(\mathtt{g})$-module $N_{\overline{1}}$ and back. 

Assume that $\mathcal{L}_{\lambda}\otimes_{U(\mathtt{g})}L$ is 
not simple and $N$ is a proper subsupermodule of
$\mathcal{L}_{\lambda}\otimes_{U(\mathtt{g})}L$.
Then from the above we have 
$\mathrm{Res}_{\mathtt{g}}^{\mathfrak{q}}\, N=L\oplus L$ or
$\mathrm{Res}_{\mathtt{g}}^{\mathfrak{q}}\, 
N=\mathrm{T}_{\lambda}^{\lambda'}L \oplus \mathrm{T}_{\lambda}^{\lambda'}L$.
In the case $\mathrm{Res}_{\mathtt{g}}^{\mathfrak{q}}
\, N=L\oplus L$ we obtain that $N$ is annihilated by
$\mathtt{c}-(\lambda_1-\lambda_2+1)^2$ as
$L\in \mathtt{Irr}^{\mathtt{g}}_{\lambda}$. This is however not possible as
$\mathtt{c}-(\lambda_1-\lambda_2+1)^2$ does not
annihilate $L(\mathcal{V}(\lambda))$ by Lemma~\ref{lem4} and hence
$\mathtt{c}-(\lambda_1-\lambda_2+1)^2\not\in \mathcal{I}_{\lambda}$.
In the case $\mathrm{Res}_{\mathtt{g}}^{\mathfrak{q}}\, 
N=\mathrm{T}_{\lambda}^{\lambda'}L\oplus 
\mathrm{T}_{\lambda}^{\lambda'}L$ we obtain a similar contradiction
using the element $\mathtt{c}-(\lambda_1-\lambda_2-1)^2$.
This proves that $\mathcal{L}_{\lambda}\otimes_{U(\mathtt{g})}{}_-$
gives a well-defined and injective 
map from $\mathtt{Irr}^{\mathtt{g}}_{\lambda}$ to  $\mathrm{Irr}_{\lambda}$.

The rest is similar to the proof of 
Proposition~\ref{prop10}\eqref{prop10.1}. For any $N\in 
\mathrm{Irr}_{\lambda}$ we fix some simple $\mathtt{g}$-submodule
$L$ of $\mathrm{Res}\, N$ and consider $\mathcal{L}(L,N)$. 
Changing, if necessary, the parity of $N$ and using
Corollary~\ref{corHC} we get $\mathcal{L}(L,N)\cong \mathcal{L}_{\lambda}$
and, finally, $\mathcal{L}_{\lambda}\otimes_{U(\mathtt{g})}L\cong N$.
This completes the proof.
\end{proof}

\subsection{Typical regular integral supermodules}\label{s3.6} 

This case splits in two subcases with different
formulations of the main result. In the first subcase we have
the same result as in the previous subsection, but a
rather different argument.

\begin{proposition}\label{prop31}
Assume that $\lambda$ is typical, regular, integral
and that $\lambda_1-\lambda_2\neq -1$.
\begin{enumerate}[(i)]
\item\label{prop31.1}
The following correspondence is a bijection between 
$\mathtt{Irr}^{\mathtt{g}}_{\lambda}$ and $\mathrm{Irr}_{\lambda}$:
\begin{displaymath}
\begin{array}{rcl}
\mathtt{Irr}^{\mathtt{g}}_{\lambda}&\leftrightarrow&\mathrm{Irr}_{\lambda}\\
L&\mapsto&\displaystyle
 \mathcal{L}_{\lambda}\otimes_{U(\mathtt{g})}L
\end{array}
\end{displaymath}
\item\label{prop31.2}
For every simple $L\in \mathtt{Irr}^{\mathtt{g}}_{\lambda}$ the module 
$\mathrm{Res}_{\mathtt{g}}^{\mathfrak{q}}\, 
\mathcal{L}_{\lambda}\otimes_{U(\mathtt{g})}L$ is 
semi-simple and we have
\begin{displaymath}
\mathrm{Res}_{\mathtt{g}}^{\mathfrak{q}}\, 
\mathcal{L}_{\lambda}\otimes_{U(\mathtt{g})}L\cong
L\oplus \mathrm{T}_{\lambda}^{\lambda'}L \oplus L
\oplus \mathrm{T}_{\lambda}^{\lambda'}L,
\end{displaymath}
where $\mathrm{T}_{\lambda}^{\lambda'}L$ is a simple module.
Moreover, we also have the isomorphism 
$\mathrm{Res}\,\mathcal{L}_{\lambda}\otimes_{U(\mathtt{g})}L\cong
L\oplus \mathrm{T}_{\lambda}^{\lambda'}L$.
\end{enumerate}
\end{proposition}

\begin{proof}
Similarly to the proof of Proposition~\ref{prop21}\eqref{prop21.2}
we have the decomposition \eqref{eq11}, where the first two direct 
summands are isomorphic to  $U(\mathtt{g})/\mathtt{I}_{\lambda}$
and result into the direct summand $L\oplus L$ of
$\mathrm{Res}\, \mathcal{L}_{\lambda}\otimes_{U(\mathtt{g})}L$.

Let us look at the summand 
$\mathcal{L}(L^{\mathtt{g}}(\lambda),L^{\mathtt{g}}(\lambda'))$.
Under our assumptions on $\lambda$ we have $\lambda_1-\lambda_2\in
\{-2,-3,-4,\dots\}$. Applying the left exact bifunctor 
$\mathcal{L}({}_-,{}_-)$ from the short exact sequence
\eqref{eq301} to the short exact sequence
\begin{equation}\label{eq321}
0\to L^{\mathtt{g}}(\lambda')\to
M^{\mathtt{g}}(-\lambda'-\alpha)\to
L^{\mathtt{g}}(-\lambda'-\alpha)\to 0
\end{equation}
we obtain the following commutative diagram with exact rows and columns:
\begin{displaymath}
\xymatrix{
\text{\tiny$\mathcal{L}(L^{\mathtt{g}}(-\lambda-\alpha),
L^{\mathtt{g}}(\lambda'))$}
\ar@{^{(}->}[r]\ar@{^{(}->}[d]
&\text{\tiny$\mathcal{L}(L^{\mathtt{g}}(-\lambda-\alpha),M^{\mathtt{g}}
(-\lambda'-\alpha))$}\ar[r]\ar@{^{(}->}[d]
&\text{\tiny$\mathcal{L}(L^{\mathtt{g}}(-\lambda-\alpha),L^{\mathtt{g}}
(-\lambda'-\alpha))$}\ar@{^{(}->}[d]\\
\text{\tiny$\mathcal{L}(M^{\mathtt{g}}(-\lambda-\alpha),
L^{\mathtt{g}}(\lambda'))$}
\ar@{^{(}->}[r]\ar[d]
&\text{\tiny$\mathcal{L}(M^{\mathtt{g}}(-\lambda-\alpha),M^{\mathtt{g}}
(-\lambda'-\alpha))$}\ar[r]\ar[d]
&\text{\tiny$\mathcal{L}(M^{\mathtt{g}}(-\lambda-\alpha),L^{\mathtt{g}}
(-\lambda'-\alpha))$}\ar[d]\\
\text{\tiny$\mathcal{L}(L^{\mathtt{g}}(\lambda),
L^{\mathtt{g}}(\lambda'))$}\ar@{^{(}->}[r]
&\text{\tiny$\mathcal{L}(L^{\mathtt{g}}(\lambda),M^{\mathtt{g}}
(-\lambda'-\alpha))$}\ar[r]
&\text{\tiny$\mathcal{L}(L^{\mathtt{g}}(\lambda),L^{\mathtt{g}}
(-\lambda'-\alpha))$}
}
\end{displaymath}
As both $L^{\mathtt{g}}(-\lambda-\alpha)$ and 
$L^{\mathtt{g}}(-\lambda'-\alpha)$ are finite dimensional while
$F$ acts injectively on all other modules involved, we obtain
\begin{displaymath}
\mathcal{L}(L^{\mathtt{g}}(-\lambda-\alpha),L^{\mathtt{g}}(\lambda'))=
\mathcal{L}(L^{\mathtt{g}}(-\lambda-\alpha),M^{\mathtt{g}}(\lambda'))=0
\end{displaymath}
By dual arguments we also have $\mathcal{L}(L^{\mathtt{g}}(\lambda),L^{\mathtt{g}}(-\lambda'-\alpha))=0$.
Hence, using arguments similar to the proof of
Lemma~\ref{lem555} we obtain the following commutative
diagram of $U(\mathtt{g})\text{-}U(\mathtt{g})$-bimodules:
\begin{displaymath}
\xymatrix{
\mathcal{L}(M^{\mathtt{g}}(-\lambda-\alpha),
L^{\mathtt{g}}(\lambda'))
\ar@{^{(}->}[r]\ar@{^{(}->}[d]
&\mathcal{L}(M^{\mathtt{g}}(-\lambda-\alpha),M^{\mathtt{g}}
(-\lambda'-\alpha))\ar@{^{(}->}[d]\\
\mathcal{L}(L^{\mathtt{g}}(\lambda),
L^{\mathtt{g}}(\lambda'))\ar[r]^{\sim}\ar@{->>}[d]
&\mathcal{L}(L^{\mathtt{g}}(\lambda),M^{\mathtt{g}}
(-\lambda'-\alpha))\\
B&&
},
\end{displaymath}
where $B$ is finite dimensional. This implies that there exists 
a short exact sequence of $U(\mathtt{g})\text{-}U(\mathtt{g})$-bimodules
\begin{displaymath}
0\to \mathcal{L}(M^{\mathtt{g}}(-\lambda-\alpha),M^{\mathtt{g}}
(-\lambda'-\alpha))\to  \mathcal{L}(L^{\mathtt{g}}(\lambda),
L^{\mathtt{g}}(\lambda'))\to B'\to 0,
\end{displaymath}
where $B'$ is finite dimensional. Tensoring the latter sequence with
$L$ and using the right exactness of the tensor product 
gives us the following exact sequence:
\begin{multline*}
\mathcal{L}(M^{\mathtt{g}}(-\lambda-\alpha),M^{\mathtt{g}}
(-\lambda'-\alpha))\otimes_{U(\mathtt{g})}L
\to \\ \to \mathcal{L}(L^{\mathtt{g}}(\lambda),
L^{\mathtt{g}}(\lambda'))\otimes_{U(\mathtt{g})}L
\to B'\otimes_{U(\mathtt{g})}L\to 0.
\end{multline*}
Similarly to the proof of Lemma~\ref{lem555} one shows that 
$B'\otimes_{U(\mathtt{g})}L=0$ as $B'$ is finite dimensional while
$L$ is infnite-dimensional and simple. This gives us a surjection
\begin{equation}\label{eq112}
\mathcal{L}(M^{\mathtt{g}}(-\lambda-\alpha),M^{\mathtt{g}}
(-\lambda'-\alpha))\otimes_{U(\mathtt{g})}L
\twoheadrightarrow  \mathcal{L}(L^{\mathtt{g}}(\lambda),
L^{\mathtt{g}}(\lambda'))\otimes_{U(\mathtt{g})}L.
\end{equation}
However, now both $M^{\mathtt{g}}(-\lambda-\alpha)$ and
$M^{\mathtt{g}}(-\lambda'-\alpha)$ are projective in $\mathtt{O}$
and hence from \cite[3.3]{BG} we have that the functor
\begin{displaymath}
\mathcal{L}(M^{\mathtt{g}}(-\lambda-\alpha),M^{\mathtt{g}}
(-\lambda'-\alpha))\otimes_{U(\mathtt{g})}{}_-
\end{displaymath}
is a projective functor, isomorphic to $\mathrm{T}_{\lambda}^{\lambda'}$.
It follows that the supermodule $\mathcal{L}(M^{\mathtt{g}}(-\lambda-\alpha),M^{\mathtt{g}}
(-\lambda'-\alpha))\otimes_{U(\mathtt{g})}L$ is simple and hence
the surjection \eqref{eq112} is, in fact, an isomorphism.
This complets the proof of \eqref{prop31.2}.

The proof of the claim \eqref{prop31.1} is now completed similarly 
to the proof of Proposition~\ref{prop10}\eqref{prop10.1}
and Proposition~\ref{prop21}\eqref{prop21.1}. 
\end{proof}

This second subcase requires a different formulation:

\begin{proposition}\label{prop32}
Assume that $\lambda$ is typical and 
$\lambda_1-\lambda_2= -1$.
\begin{enumerate}[(i)]
\item\label{prop32.1}
The following correspondence is a bijection between 
$\mathtt{Irr}^{\mathtt{g}}_{\lambda'}$ and $\mathrm{Irr}_{\lambda}$:
\begin{displaymath}
\begin{array}{rcl}
\mathtt{Irr}^{\mathtt{g}}_{\lambda'}&\leftrightarrow&\mathrm{Irr}_{\lambda}\\
L&\mapsto&\displaystyle
 \mathcal{M}_{\lambda}\otimes_{U(\mathtt{g})}L
\end{array}
\end{displaymath}
\item\label{prop32.2}
For every simple $L\in \mathtt{Irr}^{\mathtt{g}}_{\lambda'}$ the module 
$\mathrm{Res}_{\mathtt{g}}^{\mathfrak{q}}\, 
\mathcal{M}_{\lambda}\otimes_{U(\mathtt{g})}L$ is 
semi-simple and we have
\begin{displaymath}
\mathrm{Res}_{\mathtt{g}}^{\mathfrak{q}}\, 
\mathcal{M}_{\lambda}\otimes_{U(\mathtt{g})}L\cong
L\oplus \mathrm{T}_{\lambda'}^{\lambda}L \oplus L
\oplus \mathrm{T}_{\lambda'}^{\lambda}L,
\end{displaymath}
where $\mathrm{T}_{\lambda'}^{\lambda}L$ is a simple module.
Moreover, we also have the isomorphism 
$\mathrm{Res}\,\mathcal{M}_{\lambda}\otimes_{U(\mathtt{g})}L\cong
L\oplus \mathrm{T}_{\lambda'}^{\lambda}L$.
\end{enumerate}
\end{proposition}

\begin{proof}
Under our assumptions we have $L^{\mathtt{g}}(\lambda)\cong
M^{\mathtt{g}}(\lambda)$ and hence a simplified version of the proof
of Proposition~\ref{prop31}\eqref{prop31.2} gives the direct sum 
decomposition from \eqref{prop32.2}. Under our assumption the weight
$\lambda$ lies on the wall and hence the 
functor $\mathrm{T}_{\lambda'}^{\lambda}$ is a
translation to the wall and thus sends simple $\mathtt{g}$-modules to 
simple $\mathtt{g}$-modules (see \cite[Proposition~3.1]{BeGi}). 
This proves the  claim \eqref{prop32.2}.

The proof of the claim \eqref{prop32.1} is similar 
to the proof of Proposition~\ref{prop10}\eqref{prop10.1}
and Proposition~\ref{prop21}\eqref{prop21.1}. 
\end{proof}

\begin{remark}\label{rem33}
{\em  
The fact that translation to the wall sends simple $\mathtt{g}$-modules 
to  simple $\mathtt{g}$-modules can be proved in a more elementary way
than \cite[Proposition~3.1]{BeGi} (where a much more general result is
established). The module $\mathrm{T}_{\lambda'}^{\lambda}L$ is nonzero
and has finite length and hence there is a simple submodule 
$L'$ of it. By adjunction of translations to the wall and out of the
wall we get 
\begin{displaymath}
0\neq \mathrm{Hom}_{\mathtt{g}}(L',\mathrm{T}_{\lambda'}^{\lambda}L)
=\mathrm{Hom}_{\mathtt{g}}(\mathrm{T}_{\lambda}^{\lambda'}L',L),
\end{displaymath}
which yields that $L$ is a simple quotient of
$\mathrm{T}_{\lambda}^{\lambda'}L'$. At the same time 
$\mathrm{T}_{\lambda'}^{\lambda}\mathrm{T}_{\lambda}^{\lambda'}\cong
\mathrm{Id}\oplus\mathrm{Id}$ by the classification of projective functors
(\cite[3.3]{BG}). From this and the exactness of 
$\mathrm{T}_{\lambda'}^{\lambda}$
it follows that $\mathrm{T}_{\lambda'}^{\lambda}L\cong L'$.
}
\end{remark}

\subsection{Typical singular supermodules}\label{s3.7} 

Here we deal with the case when $\lambda=(t,t)$, $t\in\Bbbk$,
$t\neq 0$. This turns out to be the most complicated case, in which
we are able to get the least amount of 
information about the corresponding simple $\mathfrak{q}$-supermodules.
We let $\mathrm{T}_{\lambda'}:\mathfrak{C}_{\lambda'}
\to \mathfrak{C}_{\lambda'}$ be the translation functor through the 
wall, which is isomorphic to the indecomposable projective functor
on $\mathtt{O}$, which sends the dominant Verma module
$M^{\mathtt{g}}(-\lambda'-\alpha)$ to the indecomposable projective
cover of $L^{\mathtt{g}}(\lambda')$, see \cite[3.3]{BG}.

\begin{proposition}\label{prop41}
Assume that $\lambda=(t,t)$, $t\in\Bbbk$, $t\neq 0$.
\begin{enumerate}[(i)]
\item\label{prop41.1}
The following correspondence is a bijection between 
$\mathtt{Irr}^{\mathtt{g}}_{\lambda'}$ and $\mathrm{Irr}_{\lambda}$:
\begin{displaymath}
\begin{array}{rcl}
\mathtt{Irr}^{\mathtt{g}}_{\lambda'}&\leftrightarrow&\mathrm{Irr}_{\lambda}\\
L&\mapsto&\displaystyle
 \mathcal{M}_{\lambda}\otimes_{U(\mathtt{g})}L
\end{array}
\end{displaymath}
\item\label{prop41.2}
For every simple $L\in \mathtt{Irr}^{\mathtt{g}}_{\lambda'}$ we have  
$\mathrm{Res}_{\mathtt{g}}^{\mathfrak{q}}\, \mathcal{M}_{\lambda}\otimes_{U(\mathtt{g})}L\cong 
Y\oplus Y$, where $Y$ is an indecomposable $\mathtt{g}$-module with simple 
top isomorphic to $L$, simple socle isomorphic to $L$ and such that the
homology of the sequence $L\hookrightarrow Y\twoheadrightarrow L$ 
is finite dimensional (and is a direct sum of several, possibly zero, 
copies of the trivial $\mathtt{g}$-module). We also have
$\mathrm{Res}\, \mathcal{M}_{\lambda}\otimes_{U(\mathtt{g})}L\cong Y$.
\end{enumerate}
\end{proposition}

\begin{proof}
As usual, we start with the claim \eqref{prop41.2}.
First we claim that the $\mathtt{g}$-modules 
$M(\mathcal{V}(\lambda))_{\overline{0}}$ and
$M(\mathcal{V}(\lambda))_{\overline{1}}$
are indecomposable. They are isomorphic via the action of 
$\overline{H}_1+\overline{H}_2$ as $\lambda$ is typical.
If $X:=M(\mathcal{V}(\lambda))_{\overline{0}}$ would be decomposable, then
$X\cong M^{\mathtt{g}}(\lambda)\oplus
M^{\mathtt{g}}(\lambda')$ by Lemma~\ref{lem3}, which would yield that 
$E$ annihilates all elements of weight $\lambda'$ in 
$M(\mathcal{V}(\lambda))$. As $\overline{E}$ must annihilate at least 
two such elements (since we have four linearly independent elements of
weight $\lambda'$ and only two linearly independent elements of
weight $\lambda$ in $M(\mathcal{V}(\lambda))$), we would have a nonzero
highest weight vector of weight $\lambda'$ in $M(\mathcal{V}(\lambda))$,
which would contradict the fact that 
the supermodule $M(\mathcal{V}(\lambda))$ is simple (see Lemma~\ref{lem4}).
This shows that $X$ is indecomposable
and hence projective in $\mathtt{O}$ by Lemma~\ref{lem3} and 
\cite[Section~5]{Ma}.

Applying $\mathcal{L}({}_-,X)$ to the short exact sequence \eqref{eq321}
and using the same arguments as in the proof of Lemma~\ref{lem555} we
obtain an exact sequence of $U(\mathfrak{q})\text{-}U(\mathtt{g})$-bimodules
\begin{displaymath}
0\to
\mathcal{L}(M^{\mathtt{g}}(-\lambda'-\alpha),X)\to
\mathcal{L}(L^{\mathtt{g}}(\lambda'),X)\to B\to 0,
\end{displaymath}
where $B$ is finite dimensional. 

Let now $L\in \mathtt{Irr}^{\mathtt{g}}_{\lambda'}$. Similarly to the proof of 
\eqref{eq4} one shows that $B\otimes_{U(\mathtt{g})}L=0$ and hence,
tensoring the above sequence with $L$, we get
\begin{displaymath}
\mathrm{Ker}
\to\mathcal{L}(M^{\mathtt{g}}(-\lambda'-\alpha),X)\otimes_{U(\mathtt{g})}L\to
\mathcal{L}(L^{\mathtt{g}}(\lambda'),X)\otimes_{U(\mathtt{g})}L \to 0.
\end{displaymath}
Again by the same arguments as in the proof of Lemma~\ref{lem555} we get that
$\mathrm{Ker}$  is finite dimensional. 

As both modules $M^{\mathtt{g}}(-\lambda'-\alpha)$ and $X$ are projective in 
$\mathtt{O}$, the functor $\mathcal{L}(M^{\mathtt{g}}(-\lambda'-\alpha),X)
\otimes_{U(\mathtt{g})}{}_-$ is a projective functor by \cite[3.3]{BG},
more precisely, the translation functor $\mathrm{T}_{\lambda'}$ through 
the wall. As $\mathrm{T}_{\lambda'}$ is self-adjoint and annihilates 
finite dimensional modules, the image of $\mathrm{T}_{\lambda'}$ does 
not contain any nontrivial finite dimensional  submodules, which yields $\mathrm{Ker}=0$ and thus
$\mathcal{L}(L^{\mathtt{g}}(\lambda'),X)
\otimes_{U(\mathtt{g})}L\cong \mathrm{T}_{\lambda'} L=:Y$.

As $L$ is simple, the standard properties of $\mathrm{T}_{\lambda'}$ (see e.g. 
\cite[3.6]{GJ}) say that  $\mathrm{T}_{\lambda'} L$  has a simple socle 
isomorphic to $L$, a  simple top isomorphic to $L$, and that 
$\mathrm{T}_{\lambda'}$ kills the homology of the complex 
$L\hookrightarrow Y\twoheadrightarrow L$, which means that
this homology is finite dimensional. The claim \eqref{prop41.2} follows.

The proof of the claim \eqref{prop41.1} is similar 
to the proof of Proposition~\ref{prop10}\eqref{prop10.1}
and Proposition~\ref{prop21}\eqref{prop21.1}. 
\end{proof}

As we see, Proposition~\ref{prop41} does not describe the structure
of the module $\mathrm{Res}\, N$, $N\in\mathrm{Irr}_{\lambda}$ for
typical singular $\lambda$ completely (in the sense that the homology 
of the sequence from Proposition~\ref{prop41}\eqref{prop41.2} does
heavily depend on the choice of the module $L$). The information about
$\mathrm{Res}\, N$, which is obtained in Proposition~\ref{prop41},
is known as the {\em rough structure} of the module
$\mathrm{Res}\, N$, see \cite{KM,MS} for details. 

\subsection{Parity change}\label{s3.8} 

In this subsection we prove the last part of our main 
Theorem~\ref{mthm1}.

\begin{proposition}\label{prop51}
For every strongly typical or atypical simple 
$\mathfrak{q}$-\-su\-per\-mo\-du\-le $N$ we have $N\not\cong \Pi N$. 
For every typical simple $\mathfrak{q}$-supermodule $N$, which is 
not strongly typical, we have $N\cong \Pi N$. 
\end{proposition}

\begin{proof}
The claim is trivial for the trivial supermodule $N$ and the
corresponding $\Pi N$. Let us assume first that $N$ is a nontrivial 
atypical supermodule and consider the action of  the element 
$\overline{H}_1+\overline{H}_2$ on $N$. As we have seen in 
Subsection~\ref{s3.4}, this action defines 
a $\mathtt{g}$-homomorphism from $N_{\overline{0}}$ to 
$N_{\overline{1}}$, and a $\mathtt{g}$-homomorphism from 
$N_{\overline{0}}$ to  $N_{\overline{1}}$. One of these homomorphisms is
zero while the other one is an isomorphism. Changing the parity swaps
these two maps and proves the claim in the case of atypical supermodules.

For strongly typical supermodules we can distinguish $N$ and $\Pi N$ via
the action of the anticenter of $U(\mathfrak{q})$. By \cite[Section~10]{Go2},
the algebra $U(\mathfrak{q})$ contains a unique up to scalar element
$T_{\mathfrak{q}}\in U(\mathfrak{q})_{\overline{0}}$, 
which commutes with all elements from $U(\mathfrak{q})_{\overline{0}}$ 
and anticommutes with all elements from $U(\mathfrak{q})_{\overline{1}}$. 
In particular, this elements acts as a scalar, say $\tau$, on 
the $U(\mathfrak{q})_{\overline{0}}$-module $N_{\overline{0}}$ and 
thus as the scalar $-\tau$ on the
$U(\mathfrak{q})_{\overline{0}}$-module $N_{\overline{1}}$.
If $N$ is strongly typical then $\tau\neq 0$ by 
\cite[Theorem~10.3]{Go2}, which yields that $T_{\mathfrak{q}}$ acts
with the eigenvalue $-\tau\neq \tau$ on $(\Pi N)_{\overline{0}}$, implying 
$N\not\cong\Pi N$.

Let now $N$ be typical but not strongly typical with annihilator
$\mathcal{I}_{\lambda}$. Then from Subsection~\ref{s3.4} we have that 
$\mathcal{V}(\lambda)\cong \Pi\mathcal{V}(\lambda)$, which yields
\begin{displaymath}
\mathcal{L}(L,\mathcal{V}(\lambda))\cong 
\mathcal{L}(L,\Pi\mathcal{V}(\lambda))\cong \Pi
\mathcal{L}(L,\mathcal{V}(\lambda)) 
\end{displaymath}
for any $\mathtt{g}$-module $L$. From Lemma~\ref{lem9} and the proofs of 
Propositions~\ref{prop21}, \ref{prop31}, \ref{prop32} and \ref{prop41}
we have that $N$ either has the form
$\mathcal{L}(L,\mathcal{V}(\lambda))\otimes_{U(\mathtt{g})}L$
or $\Pi\mathcal{L}(L,\mathcal{V}(\lambda))\otimes_{U(\mathtt{g})}L$
for some simple $\mathtt{g}$-module $L$. The claim of the proposition
follows.
\end{proof}

Now we are ready to prove our first main theorem, namely
Theorem~\ref{mthm1}.

\begin{proof}[Proof of Theorem~\ref{mthm1}.]
The claim Theorem~\ref{mthm1}\eqref{mthm1.1} is Proposition~\ref{prop8}.
The claim Theorem~\ref{mthm1}\eqref{mthm1.2}
follows from Lemma~\ref{lem9} and  Propositions~\ref{prop10},
\ref{prop21}, \ref{prop31}, \ref{prop32} and \ref{prop41}.   
The claim Theorem~\ref{mthm1}\eqref{mthm1.3} is Proposition~\ref{prop51}.
\end{proof}

\subsection{Weight supermodules}\label{s3.9} 

Weight supermodules form a very special and important class of supermodules.
It is easy to see that all bijections between simple 
$\mathfrak{q}$-supermodules and $\mathtt{g}$-modules, described in 
Lemma~\ref{lem9} and  Propositions~\ref{prop10}, \ref{prop21}, \ref{prop31}, 
\ref{prop32} and \ref{prop41}, restrict to the corresponding subclasses of
weight (super)modules. For a classification of simple weight
$\mathtt{g}$-module we refer the reader to \cite[Chapter~3]{Ma}.
In this subsection we present an alternative
approach to the classification of simple weight 
$\mathfrak{q}$-supermodules using the coherent families
approach  from \cite{Mat}
(see \cite[Section~3.5]{Ma} for the corresponding arguments
in the case of $\mathtt{g}$-modules). For some other Lie superalgebras
analogous approach can be found in \cite{Gr}. 

For $z\in\Bbbk$ denote by ${}^zU'$ the 
$U(\mathfrak{q})\text{-}U(\mathfrak{q})$ bimodule $U'$, where 
the right action of $U(\mathfrak{q})$ is given by the usual multiplication,
while left action of $U(\mathfrak{q})$ is given by 
multiplication twisted by $\theta_z$.

\begin{proposition}\label{prop61}
Every simple weight $\mathfrak{q}$-supermodule is isomorphic to
one of the following supermodules:
\begin{enumerate}[(a)]
\item\label{prop61.1} Simple finite dimensional supermodule.
\item\label{prop61.2} Simple infinite dimensional highest weight
supermodule.
\item\label{prop61.3} Simple infinite dimensional lowest weight
supermodule.
\item\label{prop61.4} A simple supermodule of the form
$L^z:={}^{z}U'\otimes_{U(\mathfrak{q})}L$ for some $z\in\Bbbk$, 
where $L$ is a simple infinite dimensional highest weight supermodule.
\end{enumerate}
\end{proposition}

The supermodules described in Proposition~\ref{prop61}\eqref{prop61.4}
are called {\em dense} supermodules, see \cite[Chapter~3]{Ma}.

\begin{proof}
Let $N$ be a simple weight $\mathfrak{q}$-supermodule. Assume first 
that $N$ contains a nonzero vector $v$ of weight $\lambda$ such that
$E(v)=0$. If we have $\overline{E}(v)=0$, then $v$ is a highest weight vector.
Using the universal property of Verma supermodules we get an epimorphism
from either $M(\mathcal{V}(\lambda))$ or $M(\Pi\mathcal{V}(\lambda))$
to $N$ and hence $N\cong L(\mathcal{V}(\lambda))$
or $N\cong L(\Pi\mathcal{V}(\lambda))$.
If $\overline{E}(v)=w\neq 0$, we have $\overline{E}(w)=0$
as $\overline{E}^2=0$ and $E(w)=E\overline{E}(w)=\overline{E}E(w)=0$.
Hence $w$ is a highest weight vector and as above we get that 
$N$ is  a highest weight supermodule (this supermodule might 
be finite dimensional).

Similarly if $N$ contains a nonzero weight vector $v$ such that 
$F(v)=0$, we obtain that $N$ is a lowest weight supermodule.

Assume, finally, that both $E$ and $F$ act injectively on $N$.
As $\mathrm{Res}_{\mathtt{g}}^{\mathfrak{q}}\, N$ is a 
$\mathtt{g}$-module of finite length,
from \cite[Chapter~3]{Ma} we get that all weight spaces of 
$N$ are finite dimensional and hence both $E$ and $F$ act, in fact,
bijectively on $N$. Thus we can lift the $U(\mathfrak{q})$-action on
$N$ to a $U'$-action and twist the later by $\theta_z$ for any
$z\in \Bbbk$. Denote the obtained $U'$-supermodule (and also its restriction
to $U(\mathfrak{q})$) by $N^{z}$. By polynomiality of $\theta_z$ we get 
that for some $z\in \Bbbk$ the supermodule $N^{z}$ contains a non-zero
weight element $v$, annihilated by $E$. By the same arguments as in the
first paragraph of this proof, we get that $N^{z}$ contains a highest
weight vector, say $w$. Let $X$ be the $U(\mathfrak{q})$-subsupermodule of 
$N^{z}$, generated by $w$. As the action of $F$ on $N^{z}$ is injective, 
the action of $F$ on $X$ is injective as well. So, $X$ cannot be
a finite dimensional supermodule. Let $Y$ be some simple subsupermodule in the
socle of $X$. Then $Y$ is a simple infinite dimensional highest weight
supermodule (in particular, $F$ acts injectively on $Y$ as well). 

By construction, the supermodule
${}^{-z}U'\otimes_{U(\mathfrak{q})}Y$ is
a nonzero subsupermodule of $N$ and hence is isomorphic to $N$. The
claim of the proposition follows.
\end{proof}

To obtain an irredundant complete list of pairwise nonisomorphic
simple weight $\mathfrak{q}$-supermodules, one could use the
following lemma:

\begin{lemma}\label{lem52}
Let $L$ be a simple infinite dimensional highest 
weight $\mathfrak{q}$-supermodule.
\begin{enumerate}[(i)]
\item\label{lem52.1} If $z,z'\in\Bbbk$, then the supermodules
$L^z$ and ${L}^{z'}$ are isomorphic if and only if
$z-z'\in \mathbb{Z}$.  
\item\label{lem52.2}
There exists at most one coset $t+\mathbb{Z}\in\Bbbk/\mathbb{Z}$ 
such that the supermodule $L^{z}$ is not simple
if and only if $z\in \mathbb{Z}$ or $z\in t+\mathbb{Z}$.
\end{enumerate}
\end{lemma}

\begin{proof}
If $z,z'\in\Bbbk$ are such that $z-z'\not\in \mathbb{Z}$, then 
the supermodules ${L}^{z}$ and $L^{z'}$ have differrent weights
and hence are not isomorphic. If $z-z'\in \mathbb{Z}$, then 
to prove that $L^z$ and ${L}^{z'}$ are isomorphic it is enough
to check that $L^0\cong L^1$. In the latter case it is easy to check
that the map  $v\mapsto F(v)$ from $L^0$ to $L^1$ is an 
isomorphism of $\mathfrak{q}$-supermodules. This proves the claim
\eqref{lem52.1}.

To prove the claim \eqref{lem52.2} we first observe 
that if $X$ is a proper subsupermodule of  ${L}^{z}$, then $F$ cannot 
act bijectively on $X$. Indeed, if $F$ would act bijectively on $X$, 
then,  comparing the characters, we would have that  
$L\cap {}^{-z}U'\otimes_{U(\mathfrak{q})}X$ would be a proper subsupermodule 
of $L$, which is not possible as $L$ is simple.

But if $F$ does not act bijectively, then  it acts only injectively on
$X$. In this case the should exist a weight element $v\in X$, which does 
not belong  to the image of $F$ and is annihilated by $E$ 
(see \cite[Chapter~3]{Ma}). 
Similarly to the proof of Proposition~\ref{prop61} one then obtains that 
either $v$ or $\overline{E}(v)$ is a highest weight vector of $X$.
However, the eigenvalue of $\mathtt{c}$ is not affected by
$\theta_z$ and is given by a quadratic polynomial. Now the claim
\eqref{lem52.2} follows from the claim \eqref{lem52.1}.
\end{proof}

\subsection{The superalgebra $\mathfrak{pq}_2$}\label{s4.10} 

The superalgebra $\mathfrak{pq}_2$ is defined as the quotient of
the superalgebra $\mathfrak{q}$ modulo the ideal, generated by the 
central element $H_1+H_2$. Hence simple $\mathfrak{pq}_2$-supermodules are 
naturally identified with simple atypical $\mathfrak{q}$-supermodules, 
and thus a classification of all simple $\mathfrak{pq}_2$-supermodules follows
directly from Subsection~\ref{s3.3} and \ref{s3.4}.

\subsection{The superalgebra $\mathfrak{sq}_2$}\label{s4.11} 

The superalgebra $\mathfrak{sq}:=\mathfrak{sq}_2$ is defines as 
a subsuperalgebra of $\mathfrak{q}$, generated by $\mathtt{g}$ and the 
elements $\overline{E}$, $\overline{F}$, $\overline{H}_1-\overline{H}_2$. 
As $\mathfrak{sq}$ is a subsuperalgebra of $\mathfrak{q}$, we
have the natural restriction functor
\begin{displaymath}
\mathrm{Res}^{\mathfrak{q}}_{\mathfrak{sq}}:
\mathfrak{q}\text{-}\mathrm{sMod}\to
\mathfrak{sq}_2\text{-}\mathrm{sMod}.
\end{displaymath}
In this section we classify all simple $\mathfrak{sq}$-supermodules.
The classification is divided into two case, first we classify simple
typical $\mathfrak{sq}$-supermodules and then we classify simple 
atypical $\mathfrak{sq}$-supermodules.

\begin{proposition}\label{prop71}
The functor $\mathrm{Res}^{\mathfrak{q}}_{\mathfrak{sq}}$ induces
a bijection between the set of isomorphism classes (up to parity
change) of simple typical $\mathfrak{q}$-supermodules and the set of 
isomorphism classes (up to parity change) of simple typical 
$\mathfrak{sq}$-supermodules.
\end{proposition}

To prove Proposition~\ref{prop71} we would need the following lemma:

\begin{lemma}\label{lem75} 
\begin{enumerate}[(i)]
\item\label{lem75.1} Every simple $\mathfrak{sq}$-supermodule is 
a subsupermodule of the supermodule, induced from a simple 
$\mathtt{g}$-module.
\item\label{lem75.2} Every simple $\mathfrak{sq}$-supermodule is 
of finite length, when considered as a $\mathtt{g}$-module.
\item\label{lem75.3} Every simple $\mathfrak{sq}$-supermodule is 
a subsupermodule of the restriction of some simple 
$\mathfrak{q}$-supermodule.
\end{enumerate}
\end{lemma}

\begin{proof}
The claim \eqref{lem75.1} is proved analogously to Proposition~\ref{prop7}.  The claim \eqref{lem75.2} is proved analogously to Proposition~\ref{prop8}.
By the PBW theorem the algebra $U(\mathfrak{q})$ is free of rank two 
both as a left and as a right $U(\mathfrak{sq})$-supermodule.
Hence the induction functor 
$\mathrm{Ind}^{\mathfrak{q}}_{\mathfrak{sq}}:
\mathfrak{sq}\text{-}\mathrm{sMod}\to
\mathfrak{q}\text{-}\mathrm{sMod}$ is exact and for 
any simple $\mathfrak{sq}$-supermodule $L$ we have 
\begin{displaymath}
\mathrm{Res}^{\mathfrak{q}}_{\mathfrak{sq}}\circ 
\mathrm{Ind}^{\mathfrak{q}}_{\mathfrak{sq}}\, L\cong L\oplus L.
\end{displaymath}
In particular, it follows that 
$\mathrm{Ind}^{\mathfrak{q}}_{\mathfrak{sq}}\, L$ is of finite
length as a $\mathtt{g}$-module, and hence also as a 
$\mathfrak{sq}$-supermodule.

Now let $L$ be a simple $\mathfrak{sq}$-supermodule and $N$ be
some simple quotient of $\mathrm{Ind}^{\mathfrak{q}}_{\mathfrak{sq}}\, L$.
Using the adjunction we have
\begin{displaymath}
0\neq \mathrm{Hom}_{\mathfrak{q}}
(\mathrm{Ind}^{\mathfrak{q}}_{\mathfrak{sq}}\, L,N)=
\mathrm{Hom}_{\mathfrak{sq}}
(L,\mathrm{Res}^{\mathfrak{q}}_{\mathfrak{sq}}\, N).
\end{displaymath}
The claim \eqref{lem75.3} follows.
\end{proof}

\begin{proof}[Proof of Proposition~\ref{prop71}.]
Because of Lemma~\ref{lem75}\eqref{lem75.3}, to prove the claim of
Proposition~\ref{prop71} it is enough to show that the restriction
of every typical simple $\mathfrak{q}$-supermodule $N$ to
$\mathfrak{sq}$ is a simple $\mathfrak{sq}$-supermodule.

Note that, since we consider now typical supermodules, we have that the element 
$H_1+H_2\in \mathfrak{sq}$ still induces an isomorphism between 
the $\mathtt{g}$-submodules $X_{\overline{0}}$ and $X_{\overline{1}}$
for any $\mathfrak{sq}$-supermodule $X$. 

Let us first observe that the restriction of every typical 
simple highest weight $\mathfrak{q}$-supermodule $L(\mathcal{V}(\lambda))$
($\Pi L(\mathcal{V}(\lambda)$) to $\mathfrak{sq}$ is a simple
$\mathfrak{sq}$-supermodule. From the previous paragraph we have
that the highest weight weight space remains a simple supermodule
over the Cartan subsuperalgebra after restriction. Hence if the
supermodule $\mathrm{Res}^{\mathfrak{q}}_{\mathfrak{sq}}\,
L(\mathcal{V}(\lambda))$ would be not simple, it would have to 
have a non-trivial new primitive element, that is $v\neq 0$
such that $E(v)=\overline{E}(v)=0$. However, the action of 
$E$ and $\overline{E}$ remain unchanged during restriction and hence
this is not possible.

From \cite{Mu} we thus obtain that typical primitive ideals of 
$U(\mathfrak{sq})$ are annihilators of the typical supermodules
$\mathrm{Res}^{\mathfrak{q}}_{\mathfrak{sq}}\,
L(\mathcal{V}(\lambda))$. In fact, the previous paragraph 
(and Lemma~\ref{lem75}\eqref{lem75.3}) now implies the
classification of all typical simple finite dimensional 
$\mathfrak{sq}$-supermodules (they are just restrictions of the
corresponding typical simple finite dimensional 
$\mathfrak{q}$-supermodules).

Let now $N$ be some typical simple infinite dimensional 
$\mathfrak{q}$-supermodule. The restriction 
$X:=\mathrm{Res}^{\mathfrak{q}}_{\mathfrak{sq}}\, N$ has finite length 
as a $\mathtt{g}$-module, and hence also as an $\mathfrak{sq}$-supermodule 
(since $\mathtt{g}$ is a subalgebra of $\mathfrak{sq}$).

If we assume that $0\neq Y\subsetneq X$ is a proper $\mathfrak{sq}$-subsupermodule,
then it is also a $\mathtt{g}$-submodule. From our explicit description of
the $\mathtt{g}$-module structure on $N$
(see Propositions~\ref{prop21}, \ref{prop31}, \ref{prop32} 
and \ref{prop41}) we obtain that
in this case the annihilator of $Y$ contains the element $\mathtt{c}-c$
for some $c\in\Bbbk$. However, from the first part of the proof we know 
that the element $\mathtt{c}-c$ does not annihilate the corresponding 
highest weight $\mathfrak{sq}$-supermodule and thus cannot belong to the
corresponding primitive ideal. The obtained contradiction shows that
$0\neq Y\subsetneq X$ is not possible and completes the proof.
\end{proof}

To classify atypical simple $\mathfrak{sq}$-supermodules we 
consider the Lie algebra $\mathtt{a}:=\mathfrak{sl}_2$ 
(the subalgebra of $\mathtt{g}$) as a Lie subsuperalgebra of
$\mathfrak{sq}$. In this case  we
have the natural restriction functor
\begin{displaymath}
\mathrm{Res}^{\mathfrak{sq}}_{\mathtt{a}}:
\mathfrak{sq}\text{-}\mathrm{sMod}\to
\mathtt{a}\text{-}\mathrm{Mod}.
\end{displaymath}

\begin{proposition}\label{prop72}
The functor $\mathrm{Res}^{\mathfrak{sq}}_{\mathtt{a}}$ induces
a bijection between the set of isomorphism classes of simple purely even
$\mathtt{a}$-(super)modules and the set of  isomorphism classes (up to 
parity change) of simple atypical $\mathfrak{sq}$-supermodules.
\end{proposition}

\begin{proof}
Let $N$ denote some simple atypical $\mathfrak{q}$-supermodule. Then 
from Lemma~\ref{lem111} it follows that either 
$N_{\overline{0}}$ or $N_{\overline{1}}$ is an
$\mathfrak{sq}$-subsupermodule of
$\mathrm{Res}^{\mathfrak{q}}_{\mathfrak{sq}}\, N$. Call
this subsupermodule $X$ and we have 
$U(\mathfrak{sq})_{\overline{1}}X=0$, which means that 
$X$ is just an $\mathtt{a}$-module, trivially extended to
an $\mathfrak{sq}$-supermodule. It follows also that
$\mathrm{Res}^{\mathfrak{q}}_{\mathfrak{sq}}\, N/X\cong\Pi X$.
This defines a map from the set of  isomorphism classes (up to 
parity change) of simple atypical $\mathfrak{sq}$-supermodules
to the set of isomorphism classes of simple $\mathtt{a}$-modules
(considered as simple purely even $\mathtt{a}$-supermodules).

From Lemma~\ref{lem9} and Proposition~\ref{prop10} we have that
for any simple $\mathtt{a}$-module $L$ there exists a (unique)
$\mathfrak{q}$-supermodule $N$, whose restriction to 
$\mathtt{a}$ is isomorphic to $L\oplus L$. Hence the above
map is bijective. The claim of the proposition follows.
\end{proof}

\begin{remark}\label{rem73}
{\rm
The claim of Proposition~\ref{prop51} obviously extends to
simple $\mathfrak{sq}$-supermodules.
}
\end{remark}

\subsection{The superalgebra $\mathfrak{psq}_2$}\label{s4.12} 

The superalgebra $\mathfrak{psq}_2$ is defined as the quotient of
the superalgebra $\mathfrak{sq}_2$ modulo the ideal, generated by the 
central element $H_1+H_2$. Hence simple $\mathfrak{psq}_2$-supermodules are 
naturally identified with simple atypical $\mathfrak{sq}_2$-supermodules and 
thus a classification of all simple $\mathfrak{psq}_2$-supermodules follows
directly from Proposition~\ref{prop72}.

\vspace{1cm}

\noindent
Department of Mathematics, Uppsala University, Box 480, SE-75106,
Uppsala, SWEDEN, e-mail: {\tt mazor\symbol{64}math.uu.se}

\end{document}